\documentclass{amsart}
\usepackage{amsmath}
\usepackage{amsfonts}
\usepackage{amssymb}
\usepackage{epsfig}
\usepackage{verbatim}
\usepackage{mathrsfs}
\newtheorem{theorem}{Theorem}[section]

\newtheorem{lm}[theorem]{Lemma}
\newtheorem{tr}[theorem]{Theorem}
\newtheorem{cor}[theorem]{Corollary}

\newtheorem{rem}[theorem]{Remark}
\newtheorem{pr}[theorem]{Proposition}
\usepackage{color}

\newtheorem{ex}[theorem]{Example}

\newcommand{\1}{<}
\newcommand{\2}{\leq}
\newcommand{\3}{\geq}
\newcommand{\4}{>}
\begin{document}
\title{Linkage principle for ortho-symplectic supergroups}
\begin{abstract}
The purpose of the paper is to derive linkage principle for modular representations of ortho-symplectic supergroups. We follow the approach of Doty and investigate in detail the representation theory of the orthosymplectic group $OSP(2|1)$ and that of its Frobenius thickening. Using the description of flags and adjacent Borel supersubgroups we derive first  the strong linkage for the Frobenius thickening $G_rT$ of the orthosymplectic supergroup $G$ of type $SpO(2m|2n+1)$ and $SpO(2m|2n)$. Based on this, we derive the linkage principle for 
orthosymplectic supergroup $SpO(2m|2n+1)$ and $SpO(2m|2n)$.
\end{abstract}
\thanks{This publication was made possible by a NPRF award NPRP 6 - 1059 - 1 - 208 from the Qatar National Research Fund 
(a member of the Qatar Foundation). The statements made herein are solely the responsibility of the authors.}
\author{Franti\v sek Marko}
\address{The Pennsylvania State University \\ 76 University Drive, \\ Hazleton, PA 18202, \\ USA}
\email{fxm13@psu.edu}
\author{Alexandr N. Zubkov}
\address{Sobolev Institute of Mathematics, Siberian Branch of Russian Academy of Science (SORAN), Omsk, Pevtzova 13, 644043, Russia; Omsk State Polytechnic University, Mira 11, 644050, Russia}
\email{a.zubkov@yahoo.com}
\maketitle

\section*{Introduction}

The objective of the paper is to establish the linkage principle for orthosymplectic supergroups over fields of positive characteristics. The linkage principle for general linear supergroups over fields of positive characteristic was established in \cite{marzub}. In both cases the approach uses Frobenius thickenings and follows the fundamental idea of 
Doty \cite{sdoty}.
For an introduction to modular representations of orthosymplectic supergroups consult \cite{shuw}.

Linkage for supergroups over a field of characteristic zero reduces to the linkage principle for integral weights of the corresponding Lie superalgebras.
Linkage principle for various Lie superalgebras was discussed extensively in \cite{chengwang}. 

In Section 1 we fix notation regarding orthosymplectic supergroups and their superalgebras of distributions. Section 2 is devoted to the investigation of flags and the corresponding adjacent Borel supersubgroups of ortho-symplectic supergroups $SpO(2m|2n+1)$ and $SpO(2m|2n)$.
In Sections 4, 5, 6 and 7 we describe the representation theory of $SpO(2|1)$ and of all its Frobenius thickenings. More specifically, in Section 4 we describe the basis of the coordinate superalgebra, and induced and simple supermodules. In Section 5 we describe morphisms $\psi$ between induced supermodules. In Section 6 we describe simple composition factors of induced $SL(2)$-modules and, based on that, we describe simple composition factors of induced $SpO(2|1)$-supermodules. Also, we determine simple composition factors of the image, kernel and cokernel of the morphisms $\psi$.
In Section 8 we derive strong linkage for the Frobenius thickening $G_rT$ of the orthosymplectic supergroups $SpO(2m|2n+1)$ and $SpO(2m|2n)$. Finally, in Section 9 we derive the strong linkage for the orthosymplectic supergroups $SpO(2m|2n+1)$ and $SpO(2m|2n)$. 

\section{Notations}
\subsection{Ortho-symplectic supergroups}

A matrix $g=\left( \begin{array}{cc}
A & B \\
C & D
\end{array}\right)$ written in a $(k|l)$-block form represents an element $g$ of $SL(k|l)$ if and only if $Ber(g)=det(A-BD^{-1}C)det(D^{-1})=1$.

Recall the definition of and basic information about ortho-symplectic supergroups $G=SpO(2n|l)$ from \cite{shuw}. 
The ortho-symplectic supergroups are of two types $SpO(2n|2m+1)$ and $SpO(2n|2m)$.

We will consider the case $l=2m+1$ first. 
Let $V$ be a standard $SpO(2n|2m+1)$-supermodule that has a basis
\[v_1, \ldots, v_n; v_{-1}, \ldots, v_{-n}; w_{\overline{1}}, \ldots , w_{\overline{m}}; w_{-\overline{1}}, \ldots, w_{-\overline{m}}; w_{\overline{0}},\] where the parities are given as $|v_i|=1$ for $i\in\{\pm 1,\ldots , \pm n\}$ and $|w_j|=0$ for $j\in \{\pm \overline{1}, \ldots , \pm\overline{m}\}\cup\{\overline{0}\}$.
Let $I$ and $-I$ denote the sets $\{1, \ldots, n\}$ and $\{-1, \ldots, -n\}$, respectively. Similarily, let $J$ and $-J$ denote the sets $\{\overline{1}, \ldots, \overline{m}\}$ and $\{-\overline{1}, \ldots, -\overline{m}\}$, respectively. 

An even supersymmetric bilinear form $(.,.)$ on $V$ is defined by
$(v_i, v_{-i})=-(v_{-i}, v_i)=1$ for $1\leq i\leq n$,
$(w_{\overline{j}}, w_{-\overline{j}})=(w_{-\overline{j}}, w_{\overline{j}})=1$ for $1\leq l\leq m$,
$(w_{\overline{0}}, w_{\overline{0}})=1$,
and the form on the remaining pairs of generators vanishes. 

The matrix $J$ of this form equals 
\[\left( \begin{array}{ccccc}
0 & E_n & 0 & 0 & 0 \\
-E_n & 0 & 0 & 0 & 0 \\
0 & 0 & 0 & E_m & 0 \\
0 & 0 & E_m & 0 & 0 \\
0 & 0 & 0 & 0 & 1 
\end{array}\right),\]
where $E_n$ and $E_m$ are identity matrices of sizes $n$ and $m$, respectively.

Write a matrix $g$ in a $(2n|2m+1)$-block form and a $(n|n|m|m|1)$-block form as
\[g=\left( \begin{array}{cc}
A & B \\
C & D
\end{array}\right)=\left( \begin{array}{ccccc}
a_{I I} & a_{I,-I} & b_{I, J} & b_{I, -J} & b_{I,\overline{0}} \\
a_{-I, I} & a_{-I,-I} & b_{-I, J} & b_{-I, -J} & b_{-I,\overline{0}} \\
c_{J, I} & c_{J, -I} & d_{J, J} & d_{J, -J} & d_{J, \overline{0}} \\
c_{-J, I} & c_{-J, -I} & d_{-J, J} & d_{-J, -J} & d_{-J, \overline{0}} \\
c_{\overline{0}, I} & c_{\overline{0},-I} & d_{\overline{0}, J} & d_{\overline{0}, -J} & d_{\overline{0},\overline{0}} 
\end{array}\right),\]
where a block $x_{U, W}$ is equal to $(g_{u w})_{u\in U, w\in W}$, where $x\in\{a, b, c, d\}$ and $U, W\in\{\pm I, \pm J, \overline{0}\}$.

Denote the top diagonal block of $J$ of size $2n\times 2n$ by $J_s$ and the bottom diagonal block of size $2m+1\times 2m+1$ by $J_o$.
Then, for a superalgebra $A\in\mathsf{SAlg}_K$ over a base field $K$, the matrix $g\in SL(2n|2m+1)(A)$ belongs to $G(A)$ if and only if $g^{st}Jg=J$, where
\[g^{st}=\left( \begin{array}{cc}
A^t & C^t \\
-B^t & D^t
\end{array}\right).\]
Equivalently, the matrix $g\in G(A)$ if and only if 
\[\  A^tJ_s A+C^tJ_o C=J_s, \ A^t J_s B +C^t J_o D=0,\ -B^t J_s B +D^t J_o D=J_o \]
and  $det(A-BD^{-1}C)det(D^{-1})=1$.
We omit the relation $-B^t J_s A +D^t J_o C = 0$ because it follows from $A^t J_s B +C^t J_o D=0$.

The Lie superalgebra $Lie(G)=spo(2n|2m+1)$ consists of all $x\in gl(2n|2m+1)$ of the form 
\[
\left( \begin{array}{ccccc}
a_{I I} & a_{I,-I} & c_{-J, -I}^t & c_{J -I}^t & c_{\overline{0},-I}^t \\
a_{-I,I} & -a_{I I}^t & -c_{-J, I}^t & -c_{J I}^t & -c_{\overline{0}, I}^t \\
c_{J I} & c_{J, -I} & d_{J J} & d_{J, -J} & -d_{\overline{0}, -J}^t \\
c_{-J, I} & c_{-J, -I} & d_{-J, J} & -d_{J J}^t & -d_{\overline{0}, J}^t \\
c_{\overline{0}, I} & c_{\overline{0},-I} & d_{\overline{0}, J} & d_{\overline{0}, -J} & 0 
\end{array}\right)\]
with $a_{I,-I}, a_{-I, I}$ being symmetric and $d_{J, -J}, d_{-J, J}$ being skew-symmetric.

If $l=2m$, then 
\[J=\left( \begin{array}{cccc}
0 & E_n & 0 & 0  \\
-E_n & 0 & 0 & 0  \\
0 & 0 & 0 & E_m  \\
0 & 0 & E_m & 0   
\end{array}\right)\]
and $g\in SL(2n|2m)(A)$ belongs to $G(A)$ if and only if $g^{st}Jg=J$.
The Lie superalgebra $spo(2n|2m)$ consists of the same matrices as above with the last row and column deleted. 

Denote by $E_{ij}$, where the indices $i$ and $j$ run over the set $\{\pm 1,\ldots , \pm n\}\cup \{\pm\overline{1}, \ldots , \pm\overline{m}\}\cup\{\overline{0}\}$, the square matrix
whose only nonzero entry is $1$ and it appears at the row corresponding to the index $i$ and the column corresponding to the index $j$. 

The maximal torus $T$ of $SpO(2n|2m+1)$ consists of all matrices
\[\left( \begin{array}{ccccc}
a_{I I} & 0 & 0 & 0 & 0 \\
0 & a_{I I}^{-1} & 0 & 0 & 0 \\
0 & 0 & d_{J J} & 0 & 0 \\
0 & 0 & 0 & d_{J J}^{-1} & 0 \\
0 & 0 & 0 & 0 & 1 
\end{array}\right),\]
where $a_{I I}$ and $d_{J J}$ are diagonal matrices.
The maximal torus of $SpO(2n|2m)$ can be obtained by deleting of last row and column.

\subsection{Superalgebras of distributions}

Let $G$ be an algebraic supergroup. We recall the definition of its superalgebra of distributions $Dist(G)$.

As a superspace, $\mathsf{Dist}(G)$ coincides with $\bigcup_{k\geq 0} Dist_k(G)\subseteq K[G]^*$, where 
$Dist_k(G)=(K[G]/\mathfrak{m}^{k+1})^*$ and $\mathfrak{m}=\ker\epsilon_{K[G]}=K[G]^+$. 
Moreover, $Dist(G)$ is a supercocommutative Hopf superalgebra with a bijective antipode (see \cite{jan, shuw, zub2} for more details). For example, the multiplication of $Dist(G)$ is defined by
\[(\phi\psi)(f)=\sum(-1)^{|\psi||f_1|}\phi(f_1)\psi(f_2),\]
where $\Delta_{K[G]}(f)=\sum f_1\otimes f_2, f\in K[G]$ and $\phi, \psi\in Dist(G)$. 

Since the category of (left) $G$-supermodules coincides with the category of (right) $K[G]$-supercomodules, every $G$-supermodule $V$ has a natural structure of a (left) $Dist(G)$-supermodule given by \[\phi v=\sum(-1)^{|\phi||v_1|} v_1\phi(f_2),\]
where $v\mapsto \sum v_1\otimes f_2$ is the supercomodule map $V\to V\otimes K[G]$ and $\phi\in Dist(G)$. The corresponding functor from the category of $G$-supermodules to the category of $Dist(G)$-supermodules is an equivalence, provided $G$ is connected (see Lemma 9.4 and Lemma 9.5 from \cite{zub2}).

\subsection{Linkage}

Let $G$ be an algebraic (super)group. Two irreducible $G$-(super)modules $L$ and $L'$ are linked if there is a chain of irreducible (super)modules $L=L_0, \ldots , L_s =L'$ such that for each $1\leq i\leq s$ either $Ext^1_G (L_i, L_{i-1})\neq 0$ or $Ext^1_G (L_{i-1}, L_i)\neq 0$. 
A block $\mathscr{B}$ of $G$ consists of all irreducible (super)modules that are linked. 

\section{Adjacent Borel subsupergroups}

\subsection{G=SpO(2n$|$2m+1)}

In this section we will consider the case $G=SpO(2n|2m+1)$, where $m\geq 1$ and $n\geq 0$. Denote by $L$ the set of labels 
\[\{1, \ldots, n\}\cup \{-1, \ldots, -n\}\cup\{\overline{1}, \ldots, \overline{m}\}\cup \{-\overline{1}, \ldots, -\overline{m}\}\cup \{\overline{0}\}.\]
The action of the maximal torus $T$ of $G$ is given by
\[tv_{\pm i}=t_{i}^{\pm 1}v_{\pm i}, tw_{\pm\overline{j}}=t_{\overline{j}}^{\pm 1}w_{\pm\overline{j}} \text{ and } tw_{\overline{0}}=w_{\overline{0}}\]
for $t\in T, 1\leq i\leq n$ and $1\leq j\leq m$.

The root system $\Phi$ of the Lie superalgebra $\mathfrak{g}=Lie(G)$ consists of its even and odd parts
\[\Phi=\Phi_0\cup\Phi_1=\{\pm\delta_i\pm\delta_{i'}, \pm 2\delta_i, \pm\epsilon_{\overline{j}}\pm\epsilon_{\overline{j}'}, \pm\epsilon_{\overline{j}}\}\cup\{\pm\delta_i\pm\epsilon_{\overline{j}}, \pm\delta_i\},\]
where $1\leq i\neq i'\leq n$ and $1\leq j\neq j'\leq m$.

Then $\mathfrak{g}=Lie(T)\oplus (\oplus_{\alpha\in\Phi}\mathfrak{g}_{\alpha})$, where generators $e_{\alpha}$ of $\mathfrak{g}_{\alpha}$ are given as follows.
\[e_{\delta_i+\delta_{i'}}=E_{i, -i'}+E_{i', -i}, \ e_{-\delta_i-\delta_{i'}}=E_{-i', i}+E_{-i, i'};
\]
\[e_{\delta_i-\delta_{i'}}=E_{i, i'}-E_{-i', -i}, \ e_{\epsilon_{\overline{j}} \ -\epsilon_{\overline{j}'}}=E_{\overline{j}, \overline{j}'}-E_{-\overline{j}', -\overline{j}};\]
\[e_{\epsilon_{\overline{j}} \ +\epsilon_{\overline{j}'}}=E_{\overline{j}, -\overline{j}'}-E_{\overline{j}', -\overline{j}}, \ e_{-\epsilon_{\overline{j}} \ -\epsilon_{\overline{j}'}}=E_{-\overline{j}, \overline{j}'}-E_{-\overline{j}', \overline{j}};\]
\[e_{2\delta_i}=E_{i, -i}, \ e_{-2\delta_i}=E_{-i, i};\]
\[e_{\epsilon_{\overline{j}}}=E_{\overline{0}, -\overline{j}}-E_{\overline{j}, \overline{0}}, \ e_{-\epsilon_{\overline{j}}}=E_{\overline{0}, \overline{j}}-E_{-\overline{j}, \overline{0}};\]
\[e_{\delta_i+\epsilon_{\overline{j}}}=E_{i, -\overline{j}} \ +E_{\overline{j}, -i}, \ e_{-\delta_i-\epsilon_{\overline{j}}}=E_{-\overline{j}, i}-E_{-i, \overline{j}};
\]
\[e_{\delta_i-\epsilon_{\overline{j}}}=E_{i \overline{j}} + E_{-\overline{j}, -i}, \ e_{-\delta_i+\epsilon_{\overline{j}}}=E_{\overline{j} i}-E_{-i, -\overline{j}};\]
\[e_{\delta_i}=E_{i, 0}+E_{0, -i}, \ e_{-\delta_i}=E_{0, i}-E_{-i, 0}.\]
Let $y_{\alpha}$ denote a vector of a Chevalley basis of $\mathfrak{g}$ that belongs to $\mathfrak{g}_{\alpha}$ for $\alpha\in \Phi$ (cf. \cite{shuw}). The parity of a weight (or root) $\alpha$ is denoted by $p(\alpha)$.

From now we use the notation $x_{\pm i}=v_{\pm i}$ for $1\leq i\leq n$ and $x_{\pm\overline{j}}=w_{\pm\overline{j}}$ for $1\leq j\leq m$. 
We also use the notation $\pi_{\pm i}=\pm\delta_i$ for $1\leq i\leq n$ and $\pi_{\pm\overline{j}}=\pm\epsilon_j$ for $1\leq j\leq m$.

Let $A$ be a subset of $L\setminus\{\overline{0}\}$ such exactly one of each indices $\{i,-i\}$ and exactly one of each indices $\{\overline{j}, -\overline{j}\}$ belongs to $A$. 
Put the elements of $A$ in some order, say $a_1, a_2, \ldots, a_{n+m}$.
Then the flag 
\[V_1(A)\subseteq V_2(A)\subseteq\ldots \subseteq V_{m+n}(A),\]
where $V_i(A)=\sum_{1\leq s\leq i} Kx_{a_s}$ for $1\leq i\leq m+n$, is a maximal isotropic flag in $V$ that will be denoted by $<a_1, a_2, \ldots , a_{m+n}>$.

Up to the conjugation by an element from $G(K)$, each Borel supersubgroup of $G$ coincides with $Stab_G(\mathcal{F})$ for a maximal isotropic flag $\mathcal{F}$ as above.
The Borel supersubgroup $B$ defined by the flag $\mathcal{F}=<a_1, \ldots, a_{m+n}>$ will be denoted by $B=B^+(\mathcal{F})$. The opposite Borel supersubgroup $B^-(\mathcal{F})$ coincides with $B^+(-\mathcal{F})$, where $-\mathcal{F}=<-a_1, \ldots , -a_{m+n}>$. If we denote by $\Phi^+(\mathcal{F})$ the subset of weights of $B^+(\mathcal{F})$, 
then $\Phi=\Phi^+(\mathcal{F})\sqcup\Phi^-(\mathcal{F})$, where $\Phi^-(\mathcal{F})=\Phi^+(-\mathcal{F})=-\Phi^+(\mathcal{F})$.

For example, the {\it standard positive} Borel supersubgroup $B^+$, whose weights form a positive part
\[\Phi^+=\{\delta_i\pm\delta_{i'}, i < i'; 2\delta_i; \epsilon_{\overline{j}}\pm\epsilon_{\overline{j}'}, j < j'; \epsilon_{\overline{j}}\}\cup\{\delta_i\pm\epsilon_{\overline{j}}; \delta_i\}\]
of $\Phi$, corresponds to the flag $\mathcal{F}^+=<1, \ldots , n, \overline{1}, \ldots , \overline{m}>$. 

Symmetrically, the {\it standard negative} Borel supersubgroup $B^-$ corresponds to the flag
$\mathcal{F}^- =<-1, \ldots, -n, -\overline{1}, \ldots, -\overline{m}>$. The weights of $B^-$ form a negative part
$\Phi^-=-\Phi^+$ of $\Phi$.

Fix a maximal isotropic flag $\mathcal{F}=<a_1, \ldots, a_{m+n}>$. A {\it subflag} $\mathcal{G}$ of $\mathcal{F}$ has a form
\[V_{k_1}(A)\subseteq V_{k_2}(A)\subseteq\ldots\subseteq V_{k_l}(A),
\] where $1\leq k_1 < k_2 <\ldots < k_l\leq m+n$. The supersubgroup $Stab_G(\mathcal{G})$, called a {\it parabolic} supersubgroup of $G$, corresponding to the 
subflag $\mathcal{G}$ of $\mathcal{F}$, will be denoted by $P(\mathcal{G})$.

If $\mathcal{F}$ is fixed, then $\mathcal{G}$ uniquely corresponds to the ordered listing $(a_{k_1}, \ldots, a_{k_l})$. 
In this case we use the notation $<a_{k_1}, \ldots, a_{k_l}>$ for $\mathcal{G}$.
Define the odd and even parts of the flag $\mathcal{G}$ by $\mathcal{G}_1=\mathcal{G}\cap\{\pm 1, \ldots, \pm n\}$ and $\mathcal{G}_0=\mathcal{G}\cap\{\pm\overline{1}, \ldots, \pm\overline{m}\}$, respectively. The orderings of $\mathcal{G}_1$ and $\mathcal{G}_0$ are naturally induced by the ordering of $\mathcal{G}$.

The largest even (super)subgroup $G_{ev}$ of $G$ is isomorphic to $Sp(2n)\times SO(2m+1)$, where the first (second) factor consists of all $g\in G_{ev}$ 
acting trivially on $V_0$ (on $V_1$, respectively). Moreover, $B^+(\mathcal{F})_{ev}=B^+(\mathcal{F}_1)\times B^+(\mathcal{F}_0)$, where
$B^+(\mathcal{F}_1)=Stab_{Sp(2n)}(\mathcal{F}_1)$ and $B^+(\mathcal{F}_0)=Stab_{SO(2m+1)}(\mathcal{F}_0)$ are Borel subgroups of $Sp(2n)$ and $SO(2m+1)$ correspondingly.

For any subflag $\mathcal{G}$ of a fixed $\mathcal{F}$, denote by $Cent_G(\mathcal{G})$ the supersubgroup of $P(\mathcal{G})$ consisting of all elements $g$ 
acting trivially on $V_{k_i}(A)/V_{k_{i-1}}(A)$ for each $1\leq i\leq r$. It is clear that $Cent_G(\mathcal{G})\unlhd P(\mathcal{G})$.

Since $Cent_G(\mathcal{F})_{ev}=Cent_{Sp(2n)}(\mathcal{F}_1)\times Cent_{SO(2m+1)}(\mathcal{F}_0)$ is the unipotent radical of the Borel subgroup
$B^+(\mathcal{F})_{ev}$, the supergroup $Cent_G(\mathcal{F})$ is unipotent (cf. \cite{zubul, amas1}). 
It is easy to see that $B^+(\mathcal{F})=T Cent_G(\mathcal{F})$, hence $Cent_G(\mathcal{F})$ is the unipotent radical of $B^+(\mathcal{F})$ and $B^+(\mathcal{F})=T\ltimes Cent_G(\mathcal{F})$. In what follows we denote $Cent_G(\mathcal{F})$ by $U^+(\mathcal{F})$. Symmetrically, $B^-(\mathcal{F})=T\ltimes U^-(\mathcal{F})$, where $U^-(\mathcal{F})=U^+(-\mathcal{F})$.

Let $\mathcal{F}=<a_1, \ldots, a_{m+n}>$ be a maximal isotropic flag and write $\mathcal{F}_1=<b_1, \ldots , b_n>$ and $\mathcal{F}_0=<c_1, \ldots , c_m>$. 
Fix an index $1\leq s < m+n$ and denote by $\mathcal{G}$ the subflag $<a_1, \ldots, a_{s-1}, a_{s+2}, \ldots, a_{m+n}>$ of $\mathcal{F}$ and by 
$\mathcal{F}'$ the flag
\[<a_1, \ldots, a_{s-1}, a_{s+1}, a_s, \ldots, a_{m+n}>.\]

\begin{lm}\label{BorelsandParabolics1}
The subset $\Phi^+(\mathcal{F})$ consists of even elements 
$\pi_{b_i}\pm\pi_{b_{i'}}$ for $i<i'$, $2\pi_{b_i}$, $\pi_{c_j}\pm\pi_{c_{j'}}$ for $j<j'$ and $\pi_{c_j}$, and 
odd elements $\pi_{b_i}+\pi_{c_j}$, $\pi_{b_i}$ and $\pm (\pi_{b_i}-\pi_{c_j})$,
where the element $\pm (\pi_{b_i}-\pi_{c_j})$ has sign $+$ if and only if $b_i$ appears before $c_j$ in $\mathcal{F}$.
Moreover, $\Phi^+(\mathcal{F}')=\Phi^+(\mathcal{F})\setminus\{\pi_{a_s}-\pi_{a_{s+1}}\}\cup\{\pi_{a_{s+1}}-\pi_{a_s}\}$ and $P(\mathcal{G})$ is the minimal parabolic supersubgroup containing both $B^+(\mathcal{F})$ and $B^+(\mathcal{F}')$.
\end{lm}
\begin{proof}
The Lie superalgebra $Lie(P)$ of a supersubgroup $P$ of $G$, containing the torus $T$, has a decomposition $Lie(P)=Lie(T)\oplus (\oplus_{\alpha\in M} \mathfrak{g}_{\alpha})$ 
for a suitable subset $M\subseteq\Phi$.
Using $Lie(Stab_G(\mathcal{F}))=Stab_{\mathfrak{g}}(\mathcal{F})$
 (see Lemma 1.2, \cite{zubul}) together with the observation that $y_{\alpha}x_{i}\neq 0$ if and only if
$\alpha+\pi_{i}\in \{\pi_j |j\in I\}$ we derive the first statement. The remaining statements are obvious.
\end{proof}

For every root $\alpha\in\Phi\setminus\{\pm\epsilon_{\overline{j}}\}$ denote by $U_{\alpha}$ a one-dimensional (even or odd) unipotent supersubgroup of $G$ such that 
\[U_{\alpha}(C)=\{\exp(c y_{\alpha})=E+c y_{\alpha}| c\in C, |c|=p(\alpha)\}\]
for each $C\in\mathsf{SAlg}_K$.
For $\alpha$ in $\{\pm\epsilon_{\overline{j}}\}$ denote 
\[U_{\alpha}(C)=\{\exp(c y_{\alpha})=E+c y_{\alpha}+\frac{1}{2}(c y_{\alpha})^2\mid c\in C_0\}.\]
\begin{lm}\label{Levi1}
The unipotent radical $UP(\mathcal{G})$ of $P(\mathcal{G})$ coincides with $Cent_{G}(\mathcal{G})=U^+(\mathcal{F})\cap U^+(\mathcal{F}')$. 

Additionally, $P(\mathcal{G})=L\ltimes UP(\mathcal{G})$, where $L$ is isomorphic to $GL(2)\times T'$ if $\pi_{a_s}-\pi_{a_{s+1}}$ is an even root, and it is isomporphic to 
$GL(1|1)\times T'$ if $\pi_{a_s}-\pi_{a_{s+1}}$ is an odd root. Here $T'$ consists of all $t\in T$ acting trivially on $V_{s+1}(A)/V_{s-1}(A)$.
\end{lm}
\begin{proof}
Set $\alpha=\pi_{a_s}-\pi_{a_{s+1}}$. It is easy to see that the Zariski closure of the subfunctor $G_{\alpha}$ of $G$ given by $C\mapsto G_{\alpha}(C)=U_{-\alpha}(C)T(C)U_{\alpha}(C)$ is a supersubgroup of $P(\mathcal{G})$. 
Moreover, $\overline{G_{\alpha}}\simeq GL(2)\times T'$ if $p(\alpha)=0$ and $\overline{G_{\alpha}}\simeq GL(1|1)\times T'$ if $p(\alpha)=1$. 

Let $L=\overline{G_{\alpha}}$. The natural supergroup morphism $P(\mathcal{G})\to GL(V_{s+1}(A)/V_{s-1}(A))$ maps $L$ onto $GL(V_{s+1}(A)/V_{s-1}(A))$, hence $P(\mathcal{G})=L Cent_G(\mathcal{G})$. Since the unipotent radical of $L$ is trivial, all statements of the lemma follows.
\end{proof}
Let $\mathcal{F}=<a_1, \ldots , a_{m+n}>$ be a maximal isotropic flag such that $a_{m+n}\in\{\pm 1, \ldots , \pm n\}$. Denote by $\mathcal{G}$ 
the subflag $<a_1, \ldots , a_{m+n-1}>$ of $\mathcal{F}$ and by $\mathcal{F}'$ the flag $<a_1, \ldots, a_{m+n-1}, -a_{m+n}>$. 

\begin{lm}\label{BorelsandParabolics2}
There is $\Phi^+(\mathcal{F}')=\Phi^+(\mathcal{F})\setminus\{\pi_{a_{m+n}}, 2\pi_{a_{m+n}}\}\cup\{-\pi_{a_{m+n}}, -2\pi_{a_{m+n}}\}$. $P(\mathcal{G})$ is the minimal parabolic supersubgroup containing both $B^+(\mathcal{F})$ and $B^+(\mathcal{F}')$.
Additionally, $UP(\mathcal{G})=U^+(\mathcal{F})\cap U^+(\mathcal{F}')$ and $P(\mathcal{G})=H\ltimes UP(\mathcal{G})$, where $H\simeq SpO(2|1)\times T'$ and $T'$ consists of all $t\in T$acting trivially on $W= Kx_{a_{m+n}}+Kx_{-a_{m+n}}+Kx_{\overline{0}}$.
\end{lm}
\begin{proof}
Set $\alpha=\pi_{a_{m+n}}$. As above, the Zariski closure of the subfunctor
\[C\mapsto G_{\alpha}(C)=
U_{-\alpha}(C)U_{-2\alpha}(C)T(C)U_{2\alpha}(C)U_{\alpha}(C)\]
is a supersubgroup of $P(\mathcal{G})$. More precisely, $\overline{G_{\alpha}}=L\times T'$, where 
\[L=\{g\in Stab_G(W)\mid g x_i=x_i, i\neq\pm a_{m+n}, i\neq\overline{0}\}\simeq SpO(2|1).\] 

Denote $\overline{G_{\alpha}}$ by $H$. For every $g\in P(\mathcal{G})$ and $i\in\{a_1, \ldots, a_{m+n-1}\}$ we have 
\[(gx_{\pm a_{m+n}}, x_i)=(x_{\pm a_{m+n}}, g^{-1}x_i)=(gx_{\overline{0}}, x_i)=(x_{\overline{0}}, g^{-1}x_i)=0.\]
In other words, $P(\mathcal{G})$ stabilizes the orthogonal sum  $V_{m+n-1}(A)\perp W$. The latter implies that there is a supergroup morphism 
$\phi : P(\mathcal{G})\to \widetilde{SpO}(W)$, where $\widetilde{SpO}(W)$ consists of all $g\in GL(W)$ preserving the induced (even) bilinear form on $W$ 
($g$ is not necessarily from $SL(W)$). Every $g\in \widetilde{SpO}(W)$ has a form $tg'$, where $t$ is a diagonal matrix from $GL(W)$ that acts trivially on all basic vectors of $W$ except $x_{\overline{0}}$ and $g'\in SpO(W)$. Therefore for every $p\in P(\mathcal{G})$ there is a $g\in L$ such that $pg^{-1}\in B^+(\mathcal{F})\cap B^+(\mathcal{F}')$, 
which implies $P(\mathcal{G})=H (U^+(\mathcal{F})\cap U^+(\mathcal{F}'))$. Since $L$ normalizes $U^+(\mathcal{F})\cap U^+(\mathcal{F}')$, the lemma follows.
\end{proof}
The proof of the next lemma is similar to the proof of Lemma \ref{BorelsandParabolics2}. Since this proof is easier, we leave it for the reader.
\begin{lm}\label{BorelsandParabolics3}
If $a_{m+n}\in\{\pm\overline{1}, \ldots, \pm\overline{m}\}$, then for $\mathcal{F}'$ and $\mathcal{G}$ as in Lemma \ref{BorelsandParabolics2}, $P(\mathcal{G})$ is the minimal parabolic supersubgroup containing both $B^+(\mathcal{F})$ and $B^+(\mathcal{F}')$. 
Additionally,
$\Phi(\mathcal{F}')=\Phi^+(\mathcal{F})\setminus\{\pi_{a_{m+n}}\}\cup\{-\pi_{a_{m+n}}\}$, 
$UP(\mathcal{G})=U^+(\mathcal{F})\cap U^+(\mathcal{F}')$,
and $P(\mathcal{G})=H\ltimes UP(\mathcal{G})$, where $H\simeq SO(3)\times T'$ and $T'$ consists of all $t\in T$ acting trivially on $W= Kx_{a_{m+n}}+Kx_{-a_{m+n}}+Kx_{\overline{0}}$.
\end{lm}

\subsection{G=SpO(2n$|$2m)}

In this section we shortly discuss the remaining case $G=SpO(2n|2m)$. 
The root system $\Phi$ of its Lie superalgebra consists of its even and odd parts
\[\Phi=\Phi_0\cup\Phi_1=\{\pm\delta_i\pm\delta_{i'}, \pm 2\delta_i, \pm\epsilon_{\overline{j}}\pm\epsilon_{\overline{j}'}\}\cup\{\pm\delta_i\pm\epsilon_{\overline{j}}\},\]
where $1\leq i\neq i'\leq n$ and $1\leq j\neq j'\leq m$. The notations for flags and Borel supersubgroups remain the same. 

We need to modify statements of previous lemmas to make to switch from $G=SpO(2n|2m+1)$ to $G=SpO(2n|2m)$ are as follows.
We need to remove all elements $\pi_{c_j}$ and $\pi_{b_i}$ from the first statement of Lemma \ref{BorelsandParabolics1}. The statement of Lemma \ref{Levi1} remains the same. In Lemma \ref{BorelsandParabolics2}, the stament should be $\Phi^+ (\mathcal{F}')=\Phi^+(\mathcal{F})\setminus\{2\pi_{a_{m+n}}\}\cup\{-2\pi_{a_{m+n}}\}$ and $H\simeq Sp(2)\times T'=SL(2)\times T'$. Finally, Lemma \ref{BorelsandParabolics3} is no longer needed because under the assumption of this lemma we have $B^+(\mathcal{F})=B^+(\mathcal{F}')$. 

\subsection{Adjacent Borel supersubgroups}
From now on any pair of Borel supersubgroups $B^+(\mathcal{F})$ and $B^+(\mathcal{F}')$ from the above lemmas will be called {\it even or odd adjacent}, 
depending on whether the corresponding root $\alpha$ is even or odd, respectively.
\begin{lm}\label{thickening}
Let $H$ be an algebraic supergroup. If $H=L\ltimes R$, then $H_l=L_l\ltimes R_l$ for every $l\geq 1$.
\end{lm}
\begin{proof}
Let $C$ denote the Hopf superalgebra $K[H]$.
The condition $H=L\ltimes R$ is equivalent to $C=D\oplus J$, where $D$ is a Hopf supersubalgebra of $C$ and $J$ is a Hopf superideal of $C$. Moreover, the epimorphism $H\to L$ is dual to the embedding $D\to C$. Since $K[H_l]=C/C (C_0^+)^{p^r}=D/D (D_0^+)^{p^r}\oplus \overline{J}$, where $\overline{J}=J/(J\cap C (C_0^+)^{p^r})$ is a Hopf superideal of $K[H_l]$, our lemma follows.
\end{proof}

\section{A note on linkage for ortho-symplectic groups}

Let $G$ be an ortho-symplectic supergroup. For a given Borel supersubgroup $B^-(\mathcal{F})$ and a weight $\lambda\in X(T)$, let $H^0_{\mathcal{F}}(\lambda^a)$ denote $ind^G_{B^-(\mathcal{F})} K^a_{\lambda}, a=0, 1$. Here $K^a_{\lambda}$ is an irreducible representation of $B^-(\mathcal{F})$ of weight $\lambda$ and of parity $a$. 
The supermodule $H^0_{\mathcal{F}}(\lambda^0)$ is denoted just by $H^0_{\mathcal{F}}(\lambda)$. 

It can be easily shown that the socle of $H^0_{\mathcal{F}}(\lambda)$ is an irreducible 
supermodule that is denoted by $L_{\mathcal{F}}(\lambda)$. Every irreducible $G$-supermodule is completely determined by its highest weight $\lambda$ and parity $a$, and is isomorphic to $L_{\mathcal{F}}(\lambda^a)$.
Moreover, all other composition factors of $H^0_{\mathcal{F}}(\lambda)$ are $L_{\mathcal{F}}(\mu^b)=\Pi^b L_{\mathcal{F}}(\lambda)$, where $\mu <_{\mathcal{F}}\lambda, p(\lambda)=p(\mu)+b$. The proof of these facts are standard and left for the reader.

Since simple supermodules are uniquely determined by their highest weights and parities and 
$Ext^1_G(L(\lambda^a), L(\mu^b)) \simeq \Pi^{a+b} Ext^1_G(L(\lambda), L(\mu))$, 
we can view blocks $\mathscr{B}$ of $G$ as subsets of dominant weights $X(T)$. 

Let $V_{\mathcal{F}}(\lambda^a)$ be a Weyl supermodule, that is $V_{\mathcal{F}}(\lambda^a)=H^0_{\mathcal{F}}(\lambda^a)^{< t >}$, where $V\mapsto V^{< t >}$ is the duality functor, induced by (super)transposition (see \cite{zub3} for more details). The following proposition is a folklore.
\begin{pr}\label{simplelinkage}
If $Ext^1_G(L_{\mathcal{F}}(\lambda), L_{\mathcal{F}}(\mu))\neq 0$, then either $L_{\mathcal{F}}(\lambda)$ (or its parity shift) is a composition factor of $H^0_{\mathcal{F}}(\mu)$, or $L_{\mathcal{F}}(\mu)$ (or its parity shift) is a composition factor of $H^0_{\mathcal{F}}(\lambda)$.
\end{pr}
\begin{proof}
Assume  
\[0\to L_{\mathcal{F}}(\mu)\to M\to \Pi^a L_{\mathcal{F}}(\lambda)\to 0,
\]
where $a=0, 1$ is a non-split short exact sequence. 

Let $v\in M$ be the highest weight vector of $\Pi ^a L_{\mathcal{F}}(\lambda)$ (modulo
$L_{\mathcal{F}}(\mu)$). Then $Dist(B^+(\mathcal{F}))^+ v\subseteq L_{\mathcal{F}}(\mu)$. Thus either $Dist(B^+(\mathcal{F}))^+ v\neq 0$, which implies 
$\lambda <_{\mathcal{F}}\mu$, 
or $Dist(B^+(\mathcal{F}))^+ v =0$. In the latter case $Dist(G)v =Dist(B^-(\mathcal{F})) v $ is mapped onto $\Pi ^a L_{\mathcal{F}}(\lambda)$, hence $Dist(G)v=M$ and 
$\mu <_{\mathcal{F}}\lambda$.

If $\lambda <_{\mathcal{F}}\mu$, then $M$ is embedded into $H^0_{\mathcal{F}}(\mu)$ (analogously to the proof of Proposition 5.6 in \cite{zub3}), hence
$\Pi ^a L_{\mathcal{F}}(\lambda)$ is a composition factor of $H^0_{\mathcal{F}}(\mu)$.

If $Dist(G)v=M$, then $M$ is an epimorphic image of $\Pi ^a V_{\mathcal{F}}(\lambda)$ (cf. \cite{zub3}), hence $\Pi ^a L_{\mathcal{F}}(\mu)$ is a composition factor of $V_{\mathcal{F}}(\lambda)^{< t >}=H^0_{\mathcal{F}}(\lambda)$.
\end{proof}
\begin{rem}\label{forWeyltoo}
The statement of the above proposition is also valid if we replace induced supermodules by Weyl modules.
\end{rem}

\section{Representations of $SpO(2|1)$}

Throughout this section $G=SpO(2|1)$.

\subsection{The coordinate algebra}

For a given $A\in\mathsf{SAlg}_K$ the group $G(A)$ consists of all matrices from $SL(2|1)(A)$
\[\left( \begin{array}{ccc}
g_{1,1} & g_{1,-1} & g_{1,\overline{0}} \\
g_{-1,1} & g_{-1,-1} & g_{-1,\overline{0}} \\
g_{\overline{0},1} & g_{\overline{0},-1} & g_{\overline{0},\overline{0}} 
\end{array}\right), \]
which satisfy
\[g_{1,1}g_{-1,-1}-g_{1,-1}g_{-1,1}+g_{\overline{0},1}g_{\overline{0},-1}=1; \ g_{\overline{0},\overline{0}}^2-2g_{1,\overline{0}}g_{-1,\overline{0}}=1;\]
\[ -g_{-1,1}g_{1,\overline{0}}+g_{1,1}g_{-1,\overline{0}}+g_{\overline{0},\overline{0}}g_{\overline{0},1}=0=-g_{-1,-1}g_{1,\overline{0}}+g_{1,-1}g_{-1,\overline{0}}+g_{\overline{0},\overline{0}}g_{\overline{0},-1}.\]
\begin{lm}\label{simplifiedform}
The superalgebra $K[G]$ is generated by the elements $x_{i,j}, x_{i, \overline{0}}$ for $i,j\in\{1,-1\}$, subject to
the relation $x_{1,1}x_{-1,-1}-x_{1,-1}x_{-1,1}+x_{1,\overline{0}}x_{-1,\overline{0}}=1$. Symmetrically, $K[G]$ is generated by the elements 
$x_{i,j}, x_{\overline{0},i}$ for $i,j\in\{1,-1\}$, subject to the relation $x_{1,1}x_{-1,-1}-x_{1,-1}x_{-1,1}+x_{\overline{0},1}x_{\overline{0},-1}=1$.
\end{lm}
\begin{proof}
The equalities
\[x_{\overline{0},1}=\frac{x_{-1,1}x_{1,\overline{0}}-x_{1,1}x_{-1,\overline{0}}}{x_{\overline{0},\overline{0}}}, \ x_{\overline{0},-1}=\frac{x_{-1,-1}x_{1,\overline{0}}-x_{1,-1}x_{-1,\overline{0}}}{x_{\overline{0},\overline{0}}}
\]
imply the relation $(x_{1,1}x_{-1,-1}-x_{1,-1}x_{-1,1})(1+x_{1,\overline{0}}x_{-1,\overline{0}})=1$. After multiplication by $1-x_{1,\overline{0}}x_{-1,\overline{0}}$ we obtain $x_{1,1}x_{-1,-1}-x_{1,-1}x_{-1,1}=1-x_{1,\overline{0}}x_{-1,\overline{0}}$.

The above formulas for $x_{\overline{0},1}$ and $x_{\overline{0},-1}$ also imply
\[\frac{x_{i,\overline{0}}x_{\overline{0},j}}{x_{\overline{0},\overline{0}}}=-x_{ij}x_{1,\overline{0}}x_{-1,\overline{0}} \text{ for } i, j\in\{1,-1\}.\]
These equalities and $Ber(x)=1$ imply $(x_{1,1}x_{-1,-1}-x_{1,-1}x_{-1,1})x_{\overline{0},\overline{0}}=1$. Since elements $x_{1,1}x_{-1,-1}-x_{1,-1}x_{-1,1}, 1+x_{1,\overline{0}}x_{-1,\overline{0}}$ and $x_{\overline{0},\overline{0}}$ commute
(they are even), the cancellation property implies $x_{\overline{0},\overline{0}}=1+x_{1,\overline{0}}x_{-1,\overline{0}}$.
From the equality $x_{\overline{0},\overline{0}}=1+x_{1,\overline{0}}x_{-1,\overline{0}}$ we deduce that
\[x_{\overline{0},1}=x_{-1,1}x_{1,\overline{0}}-x_{1,1}x_{-1,\overline{0}}, \ x_{\overline{0},-1}=x_{-1,-1}x_{1,\overline{0}}-x_{1,-1}x_{-1,\overline{0}},\]
which proves the first part of the statement.
We leave it for the reader to prove the second part of the statement using symmetrical equalities $x_{\overline{0},\overline{0}}=1+x_{\overline{0},1}x_{\overline{0},-1}$ and
\[x_{1,\overline{0}}=-x_{1,-1}x_{\overline{0},1}+x_{1,1}x_{\overline{0},-1}, \ x_{-1,\overline{0}}=-x_{-1,-1}x_{\overline{0},1}+x_{-1,1}x_{\overline{0},-1}.\]
\end{proof}

Define $x'_{1,1}=x_{1,1}(1+x_{1,\overline{0}}x_{-1,\overline{0}})$ and $x'_{1,-1}=x_{1,-1}(1+x_{1,\overline{0}}x_{-1,\overline{0}})$. 
We can express $K[G]$ in terms of generators $x'_{1,1}$, $x'_{1,-1}$, $x_{-1,1}$, $x_{-1,-1}$, $x_{1,\overline{0}}$ and $x_{-1,\overline{0}}$
and the previous relation becomes $x'_{1,1}x_{-1,-1}-x'_{1,-1}x_{-1,1}=1$. We will denote $D=x'_{1,1}x_{-1,-1}-x'_{1,-1}x_{-1,1}$. Then 
\[K[G]=K[x'_{1,1}, x'_{1,-1}, x_{-1,1}, x_{-1,-1}, x_{1,\overline{0}}, x_{-1,\overline{0}}]/<D-1>\simeq K[SL(2)]\otimes K[x_{1,\overline{0}}, x_{-1,\overline{0}}].\] 

Consider the weights with respect to the right action of the torus $T$ consisting of diagonal matrices 
$\left( \begin{array}{ccc}
g_{1,1} & 0 & 0 \\
0 & g_{1,1}^{-1} & 0  \\
0 & 0 & 1 
\end{array}\right)$. 
Since the weight of $x_{1,\overline{0}}$ is $+1$, the weight of $x_{-1,\overline{0}}$ is $-1$, and the weight of $1+x_{1,\overline{0}}x_{-1,\overline{0}}$ is zero, the new generators 
$x'_{1,1}$ and $x'_{1,-1}$ have the same weights as the corresponding old generators $x_{1,1}$ and $x_{1,-1}$, respectively.

The Lie superalgebra $Lie(G)$ of $G$ has a basis consisting of the matrices
\[{\bf h}=E_{1,1}-E_{-1,-1}, {\bf e}=E_{1,-1}, {\bf f}=E_{-1,1}, {\bf x}=E_{1,\overline{0}}+E_{\overline{0},-1}, {\bf y}=E_{-1,\overline{0}}-E_{\overline{0},1}.
\]
The distribution superalgebra $Dist(G)$ is generated by the divided powers ${\bf e}^{(t)}, {\bf f}^{(t)}$, where $t>0$, of the elements ${\bf e}, {\bf f}$, by the binomial polynomials 
$\left(\begin{array}{c} {\bf h} \\ t \end{array}\right)$ and by the elements ${\bf x}, {\bf y}$ (see \cite{shuw}).

The element $\bf f$ acts on $K[G]$ as an even right superderivation, hence  
\[(x y){\bf f}^{(t)}=\sum_{0\leq i\leq t}\left(\begin{array}{c} t \\ i \end{array}\right)x{\bf f}^{(i)} y{\bf f}^{(t-i)} \text{ for } x, y\in K[G], \]
and the element $\bf y$ acts on $K[G]$ as an odd right superderivation.

\begin{lm}\label{action}
The elements ${\bf f}^{(t)}$ and $\bf y$ act on the generators of $K[G]$ for each $i,j\in\{1,-1\}$ as follows.
\begin{enumerate}
\item $x_{i,j}{\bf f}=\delta_{i,-1}x_{1,j}$ and  $x_{i,\overline{0}}{\bf f}=\delta_{i,-1}x_{1,\overline{0}}$;
\item $x_{i,j}{\bf f}^{(t)}=x_{i,\overline{0}}{\bf f}^{(t)}=0$ for $t\geq 2$;
\item $x_{ij}{\bf y}=-\delta_{i,-1}(x_{-1,j}x_{1,\overline{0}}-x_{1, j}x_{-1,\overline{0}})$ and $x_{i,\overline{0}}{\bf y}=\delta_{i,-1}(1+x_{1,\overline{0}}x_{-1,\overline{0}})$.
\end{enumerate}
\end{lm}
\begin{proof}
Straightforward verification.
\end{proof}

\begin{lm}\label{leftaction}
The left action of the superalgebra $Dist(G)$ on the generators of $K[G]$ for each $i,j\in\{1,-1\}$ is given as follows.
\begin{enumerate}
\item ${\bf e}^{(t)}x_{i,j}=\delta_{t 1}\delta_{-1, j}x_{i, 1}, \ {\bf f}^{(t)}x_{i,j}=\delta_{t 1}\delta_{1 j}x_{i,-1}$ for $t\geq 1$;
\item ${\bf x}x_{i,j}=-\delta_{-1, j}x_{i, \overline{0}}, \ {\bf y}x_{i,j}=\delta_{1 j}x_{i, \overline{0}}$;
\item ${\bf e}^{(t)}x_{i,\overline{0}}={\bf f}^{(t)}x_{i, \overline{0}}=0$ for $t\geq 1$;
\item ${\bf x}x_{i, \overline{0}}=x_{i, 1}, \ {\bf y}x_{i,\overline{0}}=x_{i,-1}$.
\end{enumerate}
\end{lm}
\begin{proof}
Straightforward verification.
\end{proof}

\subsection{Induced and simple supermodules}

The group $G$ has only two Borel supersubgroups, $B^-$ and $B^+$. They consist of all matrices of the form
\[\left( \begin{array}{ccc}
g_{1,1} & 0 & 0 \\
g_{-1,1} & g_{1,1}^{-1} & -g_{1,1}^{-1}g_{\overline{0},1}  \\
g_{\overline{0},1} & 0 & 1 
\end{array}\right) \]
and 
\[\left( \begin{array}{ccc}
g_{1,1} & g_{1,-1} & g_{1,\overline{0}} \\
0 & g_{1,1}^{-1} & 0 \\
0 & g_{1,1}^{-1}g_{1,\overline{0}} & 1 
\end{array}\right), \]
respectively. 

Remind that every irreducible representation of $B^-$ (and/or $B^+$) is isomorphic to $K^a_l$, where
$l\in X(T)=\mathbb{Z}$ and $a\in\mathbb{Z}_2$. 

Let $H$ be an algebraic supergroup and $R$ be its connected subsupergroup.
Recall that $K[H]$ has a natural structure of a right $Dist(R)$-supermodule via $f\phi=\sum (-1)^{|\phi||f_2|}\phi(f_1)f_2$, where $\Delta(f)=\sum f_1\otimes f_2$
for $f\in K[H]$ and $\phi\in Dist(R)$.

Let $V$ be an one-dimensional $R$-supermodule given by a character $\chi\in X(R)$ such that 
$v\mapsto v\otimes\chi$ for $v\in V$.
\begin{lm}\label{induced}
The induced supermodule $ind^H_R V$ can be naturally identified with the $H$-subsupermodule of $K[H]$ consisting of those elements $f\in K[H]$ for which 
$f\phi=\phi(\chi)f$ for every $\phi\in Dist(R)$.
\end{lm}
\begin{proof}
Choose $v\in V\setminus 0$.
By definition, 
\[ind^H_R V=\{v\otimes f\in V\otimes K[H]\mid v\otimes\chi\otimes f=\sum v\otimes
\overline{f_1}\otimes f_2\},\] 
where $f\mapsto\overline{f}$ denotes the epimorphism $K[H]\to K[R]$.
It is clear that the map $v\otimes f\mapsto f$ induces an embedding of the $H$-supermodule $ind^H_R V$ into $K[H]$. Using this embedding, we identify
\[ind^H_R V=\{f\in K[H]\mid \chi\otimes f=\sum\overline{f_1}\otimes f_2\}.\] Since $R$ is connected, the last condition is equivalent to the statement of the lemma.
\end{proof}

Coming back to $SpO(2|1)$, we get
\[ind^G_{B^{\pm}} K_l^a=H_{\pm}^0(l)=\{f\in K[G]\mid f\phi=\phi(g_{1,1}^l) f  \text{ for every } \phi\in Dist(B^{\pm})\}.\]
Observe that in the notations of the previous section we have $H^0_-(l)=H^0_{< 1 >}(l)$ and $H^0_+(l)=H^0_{< -1 >}(l)$.

Let us consider the supermodule $H^0_-(l)$. The superalgebra $Dist(B^-)$ is generated by the elements $\bf{y}$, and $\left(\begin{array}{c} {\bf h} \\ t \end{array}\right), {\bf f}^{(t)}$ for $t\geq 0$.

The vector $f=\prod_{i, j\in\{1,-1\}}x_{i,j}^{k_{ij}}x_{1,\overline{0}}^{u}x_{-1,\overline{0}}^v$ is an element of the weight \[k=k_{11}+k_{12}+u-k_{21}-k_{22}-v\] since $f\left(\begin{array}{c} {\bf h} \\ t \end{array}\right)=\left(\begin{array}{c} k \\ t \end{array}\right)f$.

As indicated earlier, it will be easier to work with the generators $x_{1,1}'$, $x_{1,-1}'$, $x_{-1,1}$, $x_{-1,-1}$, $x_{1,\overline{0}}$ and $x_{-1,\overline{0}}$. The action of ${\bf f}^{(t)}$ and ${\bf y}$
on these generators is given as follows.

Since $(1+x_{1,\overline{0}}x_{-1,\overline{0}}){\bf f}^{(t)}=0$ for every $t\geq 1$, we obtain $x'_{1, j}{\bf f}^{(t)}=0$
and $x_{-1, j}{\bf f}^{(t)}=\delta_{t 1} x'_{1, j}(1-x_{1,\overline{0}}x_{-1,\overline{0}})$.
Furthermore, $(1+x_{1,\overline{0}}x_{-1,\overline{0}}){\bf y}=x_{1,\overline{0}}$ implies $x'_{1, j}{\bf y}=x'_{1, j}x_{1,\overline{0}}$. Since 
$x'_{1, j}x_{-1,\overline{0}}=x_{1, j}x_{-1,\overline{0}}$, we infer that $x_{-1, j}{\bf y}=-(x_{-1, j}x_{1,\overline{0}}-x'_{1 j}x_{-1,\overline{0}})$.

Let ${\bf k}$ denote a matrix
\[{\bf k}=\left(\begin{array}{cc}
k_{1,1} & k_{1,-1} \\
k_{-1,1} & k_{-1,1}\end{array}\right)\]
with non-negative integer coefficients, and $x^{\bf k}$ denote the element ${x'}_{1,1}^{k_{1,1}}{x'}_{1,-1}^{k_{1,-1}}x_{-1,1}^{k_{-1,1}}x_{-1,-1}^{k_{-1,-1}}$. 
For $i,j\in\{1,-1\}$ define by $\epsilon_{ij}$ that $2\times 2$-matrix whose only nonzero entry is $1$ and it appears at the position corresponding to $(i,j)$. 
Then the following general formula is valid.
\begin{equation}\label{star}
\begin{aligned}
x^{\bf k}{\bf y}&=(\sum_{j\in\{1,-1\}}k_{1, j})x^{\bf k} x_{1,\overline{0}} 
-\sum_{j\in\{1,-1\}}k_{-1,j}x^{{\bf k}-\epsilon_{-1, j}}(x_{-1, j}x_{1,\overline{0}}-x'_{1, j}x_{-1,\overline{0}})\\
&=(k_{1,1}+k_{1,-1}-k_{-1,1}-k_{1,-1})x^{\bf k} x_{1,\overline{0}}+(\sum_{j\in\{1,-1\}}k_{-1, j}x^{{\bf k}-\epsilon_{-1, j}+\epsilon_{1, j}})x_{-1,\overline{0}}.
\end{aligned}
\end{equation}
We identify the subalgebra $A=K[x'_{1,1}, x'_{1,-1}, x_{-1,1}, x_{-1,-1}]/<D-1>$ with $K[SL(2)]$. The Lie algebra of $SL(2)$ is generated by the elements ${\bf h}'=E_{11}-E_{-1,-1}$,
${\bf e}'=E_{1,-1}$ and ${\bf f}'=E_{-1,1}$.
Then ${\bf f}'$ acts on $A$ on the right by $x'_{1, j}{\bf f}'=0$ and $x_{-1, j}{\bf f}'=x'_{1 j}$, and
\[x^{\bf k} {\bf f}'=\sum_{j\in\{1,-1\}}k_{-1,j}x^{{\bf k}-\epsilon_{-1, j}+\epsilon_{1, j}}.\]

Every $f\in K[G]$ has a unique form $f=f_0+f_1 x_{1,\overline{0}}+f_{-1} x_{-1,\overline{0}} + f_{1,-1}x_{1,\overline{0}}x_{-1,\overline{0}}$, where the coefficients 
$f_0, f_1, f_{-1} , f_{1,-1}$ belong to the subalgebra $A$. Assume that $f\in H^0_-(l)$. By Lemma \ref{induced} it is the case if and only if $f$ has a weight $l$ and $f{\bf f}^{(t)}=f{\bf y}=0$ for $t\geq 1$.
\begin{lm}\label{nullspaceof y}
The element $f$ as above satisfies $f{\bf y}=0$ if and only if the following equations hold.
\begin{itemize}
\item $f_0 {\bf y}+f_{1,-1}x_{1,\overline{0}}=0$;
\item $f_{-1}=0$;
\item $(f_1 {\bf y})x_{1,\overline{0}}=0$.
\end{itemize}
\end{lm}
\begin{proof}
We have 
\[f {\bf y}=f_0 {\bf y}-(f_1 {\bf y})x_{1,\overline{0}}+f_{-1} (1+x_{1,\overline{0}}x_{-1,\overline{0}})-(f_{-1} {\bf y})x_{-1,\overline{0}}+f_{1,-1}x_{1,\overline{0}}+
(f_{1,-1} {\bf y})x_{1,\overline{0}}x_{-1,\overline{0}}=0.
\]
The formula $(\ref{star})$ implies $(f_{1,-1} {\bf y})x_{1,\overline{0}}x_{-1,\overline{0}}=0$. 

Since all remaining summands belong to $Ax_{1,\overline{0}}\oplus Ax_{-1,\overline{0}}\oplus Ax_{1,\overline{0}}x_{-1,\overline{0}}$ except for $f_{-1}$, we derive that $f_{-1}=0$. Finally, the summand $(f_1 {\bf y})x_{1,\overline{0}}$ belongs to $Ax_{1,\overline{0}}x_{-1,\overline{0}}$ and the summand $(f_0 {\bf y}+f_{1,-1}x_{1,\overline{0}})$ belongs to $Ax_{1,\overline{0}}\oplus Ax_{-1,\overline{0}}$. The lemma follows.
\end{proof}

The equation $(f_1 {\bf y})x_{1,\overline{0}}=0$ is equivalent to $f_1{\bf f}'=0$. Furthermore, the equation $f_0 {\bf y}+f_{1,-1}x_{1,\overline{0}}=0$ implies $f_0 {\bf f}'=0$.

Since $x'_{1, j}{\bf f}^{(t)}=x'_{1, j}{\bf f}'^{(t)}$ and $x_{-1, j}{\bf f}^{(t)}=x_{-1, j}{\bf f}'^{(t)}(1-x_{1,\overline{0}}x_{-1,\overline{0}})$, we have 
\begin{equation}\label{sstar}
\begin{aligned}
x^{\bf k} {\bf f}^{(t)}=x^{\bf k} {\bf f}'^{(t)}(1-x_{1,\overline{0}}x_{-1,\overline{0}})^t.
\end{aligned}
\end{equation}
Aplying formula (\ref{sstar}) to the element $f=f_0 +f_1 x_{1,\overline{0}}+f_{1,-1} x_{1,\overline{0}}x_{-1,\overline{0}}$ we obtain
\[f {\bf f}^{(t)}=(f_0 {\bf f}'^{(t)})(1-t x_{1,\overline{0}}x_{-1,\overline{0}}) +(f_1 {\bf f}'^{(t)})x_{1,\overline{0}}+
(f_{1,-1} {\bf f}'^{(t)})x_{1,\overline{0}}x_{-1,\overline{0}}=0.
\]
Thus $f_0 {\bf f}'^{(t)}=f_1 {\bf f}'^{(t)}=f_{1,-1} {\bf f}'^{(t)}=0$ for all $t\geq 1$.
The latter equations mean that $f_0, f_{1,-1}$ and $f_1$ belong to the induced modules $ind^{SL(2)}_{B'^{-}} K_l$ and $ind^{SL(2)}_{B'^{-}} K_{l-1}$ respectively, where $B'^{-}$ is a Borel subgroup of $SL(2)$ consisting of all lower triangular matrices.

The structure of arbitrary $ind^{SL(2)}_{B'^{-}} K_a$ is well known (cf. \cite{jan}, page ix). As a submodule of $K[SL(2)]$ it has a basis consisting of elements ${x'}_{1,1}^i {x'}_{1,-1}^{a-i}$ for 
$0\leq i\leq a$. The following proposition is now evident.
\begin{pr}\label{descriptionofinduced}
The supermodule $H^0_-(l)$ has a basis consisting of elements
\[{x'}_{1,1}^i {x'}_{1,-1}^{l-i}(1-lx_{1,\overline{0}}x_{-1,\overline{0}})=x_{1,1}^i x_{1,-1}^{l-i}, \ {x'}_{1,1}^j {x'}_{1,-1}^{l-1-j}x_{1,\overline{0}}=
x_{1,1}^j x_{1,-1}^{l-1-j}x_{1,\overline{0}},\]
where $0\leq i\leq l$ and $0\leq j\leq l-1$.
\end{pr}

The highest weight vector of $H^0_-(l)$ equals $x_{1,1}^l$. Since $G$ is connected, the simple socle $L_-(l)$ of $H^0_-(l)$ coincides with $Dist(G)x_{1,1}^l=Dist(U^-)x_{1,1}^l$, where $U^-$ is the unipotent radical of $B^-$. The superalgebra $Dist(U^-)$ has a basis consisting of elements  ${\bf y}$, ${\bf f}^{(t)}$ and ${\bf f}^{(t)} {\bf y}$ for $t\geq 1$.
\begin{lm}\label{soclewrtlower}
If $p|l$, then $L_-(l)$ has a basis consisting of elements $x_{1,1}^i x_{1,-1}^{l-i}$ such that $\binom{l}{i}$ is not divisible by $p$.

If $p\not|l$, then a basis of $L_-(l)$ is given by the elements $x_{1,1}^i x_{1,-1}^{l-i}$ such that $\binom{l}{i}$ is not divisible by $p$ together with elements 
of the form $x_{1,1}^j x_{1,-1}^{l-1-j}x_{1,\overline{0}}$ such that $p$ does not divide $\binom{l-1}{j}$.
\end{lm}
\begin{proof}
Use 
\[{\bf y}x_{1,1}^l=lx_{1,1}^{l-1}x_{1,\overline{0}}, \ {\bf f}^{(t)}(x_{1,1}^l)=\binom{l}{t}x_{1,1}^{l-t}x_{1,-1}^t\] 
and
\[{\bf f}^{(t)} {\bf y}(x_{1,1}^l)=l\binom{l-1}{t}x_{1,1}^{l-1-t}x_{1,-1}^t x_{1,\overline{0}}.\]
\end{proof}

The following statement is analogous to Proposition \ref{descriptionofinduced}.
\begin{pr}\label{descriptionofinduced+}
The induced supermodule $H^0_+(-l)$ has a basis consisting of elements 
$x_{-1,-1}^i x_{-1,1}^{l-i-\epsilon}x_{-1,\overline{0}}^{\epsilon}$, where $\epsilon=0, 1$ and $0\leq i\leq l-\epsilon$. 
\end{pr}
\begin{proof}
We can work with elements $x_{1,1}, x_{1,-1}$, $x'_{-1,1}= x_{-1,1}(1+x_{1,\overline{0}}x_{-1,\overline{0}})$ and $x'_{-1,-1}=x_{-1,-1}(1+x_{1,\overline{0}}x_{-1,\overline{0}})$ and rephrase our previous arguments. We leave the details to the reader. 
\end{proof}

Moreover, the simple socle $L_+(-l)$ of $H^0_+(-l)$ coincides with $Dist(U^+)x_{-1,-1}^l$ and 
the superalgebra $Dist(U^-)$ has a basis consisting of elements  ${\bf x}$, ${\bf e}^{(t)}$ and ${\bf e}^{(t)} {\bf x}$ for $t\geq 1$.
We leave it for the reader to verify the following statement.

\begin{lm}\label{soclewrtupper}
If $p|l$, then $L_+(-l)$ has a basis consisting of elements $x_{-1,-1}^i x_{-1,1}^{l-i}$ such that $\binom{l}{i}$ is not divisible by $p$.

If $p\not|l$, then a basis of $L_+(-l)$ is given by the elements $x_{-1,-1}^i x_{-1,1}^{l-i}$ such that $\binom{l}{i}$ is not divisible by $p$ together with  
elements of the form $x_{-1,-1}^j x_{-1,1}^{l-1-j}x_{-1,\overline{0}}$ such that $p$ does not divide $\binom{l-1}{j}$.
\end{lm}

\section{Morphisms between induced supermodules of $SpO(2|1)$}

Frobenius reciprocity law implies
\[Hom_G(H^0_+(-k), H^0_-(l))\simeq Hom_{B^-}(H_+^0(-k), K_{l})=\]
\[Hom_T(H_+^0(-k)/rad_{B^-} H_+^0(-k), K_l),\]
where $Hom$ stands for all homomorphisms, not only even ones.

If $v\in rad_{B^-} V=Dist(U^-)^+V$, then $v$ equals either ${\bf y}v'$ or ${\bf f}^{(t)}v'$ for an element $v'\in V$. Therefore it is enough to use the formulae
\[{\bf f}^{(t)}x_{-1,-1}^i x_{-1,1}^{k-i-\epsilon}x_{-1,\overline{0}}^{\epsilon}=\binom{k-i-\epsilon}{t}x_{-1,-1}^{i+t} x_{-1,1}^{k-i-\epsilon-t}x_{-1,\overline{0}}^{\epsilon},\]
\[{\bf y}x_{-1,-1}^i x_{-1,1}^{k-i-\epsilon}x_{-1,\overline{0}}^{\epsilon}=
(k-i-\epsilon)x_{-1,-1}^i x_{-1,1}^{k-i-1-\epsilon}x_{-1,\overline{0}}^{\epsilon+1}+\epsilon x_{-1,-1}^{i+1} x_{-1,1}^{k-i-\epsilon}\]
to verify that if $k> 0$, then the component of $rad_{B^-} H^0_+(-k)=Dist(U^-)^+ H^0_+(-k)$ belonging to $A$ is spanned by the elements
$x_{-1,-1}^i x_{-1,1}^{k-i}$ for $1\leq i \leq k$ (since ${\bf y}x_{-1,-1}^i x_{-1,1}^{k-i-1}x_{-1,\overline{0}}$ $= x_{-1,-1}^{i+1} x_{-1,1}^{k-i-1}$) 
and the component of $rad_{B^-} H^0_+(-k)$ belonging to 
$Ax_{-1,\overline{0}}$ is spanned by $k x_{-1,1}^{k-1}x_{-1,\overline{0}}$ and the following multiples of $x_{-1,-1}^j x_{-1,1}^{k-1-j}x_{-1,\overline{0}}$ 
\[\{\binom{k-j}{1}x_{-1,-1}^j x_{-1,1}^{k-1-j}x_{-1,\overline{0}}, \ldots, \binom{k-j+t-1}{t}x_{-1,-1}^j x_{-1,1}^{k-1-j}x_{-1,\overline{0}}, \ldots, \]
\[\binom{k-1}{j}x_{-1,-1}^j x_{-1,1}^{k-1-j}x_{-1,\overline{0}}\}\]
for each $1\leq j \leq k-1$.

We will denote by $p^{a(l)}$ the exact power of $p$ dividing $l$, and write $p^{a(l)}||l$.

\begin{lm}\label{Fero's responsibility}
All elements of the set $\{\binom{k-j}{1}, \ldots, \binom{k-j+t-1}{t}, \ldots, \binom{k-1}{j}\}$
are divisible by $p$ if and only if $j<p^{a(k-j)}$.
\end{lm}
\begin{proof}
By Kummer's theorem, the maximal power of $p$ that divides $\binom{n}{m}$ equals the number of carries when $m$ is added to $n-m$ in the $p$-adic base.
Let $t=p^{a(k-j)}$ and $k-j=p^{a(k-j)}l$, where $p$ does not divide $l$.
If we add $t=p^{a(k-j)}$ to $k-j-1=(l-1)p^{a(k-j)}+(p^{a(k-j)}-1)$ in the $p$-adic base, then there is no carry. This means that $\binom{k-j+p^{a(k-j)}-1}{p^{a(k-j)}}$ is not divisible by $p$. 

To prove the converse statement we use induction on $t$. First, it is clear for $t=1$. If $t>1$, then $\binom{k-j+t-1}{t}=\binom{k-j+t-1}{t-1}\frac{k-j}{t}$ and since $p^{a(k-j)}||(k-j)$ and $1\leq t  < p^{a(k-j)}$, the second factor is divisible by $p$.

\end{proof}

Using Lemma \ref{Fero's responsibility} we conclude that if $k>0$ then the basis of $rad_{B^-} H^0_+(-k)$ is given 
by elements $x_{-1,-1}^i x_{-1,1}^{k-i}$ for $1\leq i \leq k$, by $x_{-1,1}^{k-1}x_{-1,\overline{0}}$ if $k$ is not divisible by $p$, 
and by elements $x_{-1,-1}^j x_{-1,1}^{k-1-j}x_{-1,\overline{0}}$ for $1\leq j\leq k-1$ such that $j\geq p^{a(k-j)}$.

Finally, if $k=0$, then $H^0_+(0)=K$ and $rad_{B^-} H^0_+(-k)$ vanishes.

The following lemma is now evident.
\begin{lm}\label{nontrivialmorphism}
Assume $k> 0$. 

If $p|k$, then $Hom_G(H^0_+(-k), H^0_-(l))\neq 0$ if and only if $l=k$ or $l=k-2j-1$, where $0\leq j\leq [\frac{k-1}{2}]$ and $j<p^{a(k-j)}$.

If $p\not| k$, then $Hom_G(H^0_+(-k), H^0_-(l))\neq 0$ if and only if $l=k$ or $l=k-2j-1$, where $1\leq j\leq [\frac{k-1}{2}]$ and $j<p^{a(k-j)}$.

If $k=0$, then $Hom_G(H^0_+(0), H^0_-(l))\neq 0$ if and only if $l=0$.

Additionally, if $Hom_G(H^0_+(-k), H^0_-(l))\neq 0$, then the dimension of this space of homomorphisms is one. This space is an odd superspace in all cases except when $l=k$.
\end{lm}
\begin{lm}\label{when k=l}
The supermodules $H^0_+(-k)$ and $H^0_-(k)$ are isomorphic to each other via the isomorphism given by 
$x_{-1,-1}^i x_{-1,1}^{k-i-\epsilon}x_{-1,\overline{0}}^{\epsilon}\mapsto x_{1,1}^{k-i-\epsilon} x_{1,-1}^i x_{1,\overline{0}}^{\epsilon}$.
\end{lm}
\begin{proof}
Proof is an exercise left for the reader.
\end{proof}
Using Lemma \ref{when k=l} we derive that $L_+(-k)\simeq L_-(k)$ for every $k\in\mathbb{Z}_{\geq 0}$.
\begin{lm}\label{towardlinkage+}
Let $0\leq j<p^{a(k-j)}$, provided $p|k$, or let $1\leq j<p^{a(k-j)}$, provided $p\not|k$. The homomorphism $\psi_{-k,k-1-2j}:H^0_+(-k)\to H^0_-(k-1-2j)$ is given on basis elements as 

\[x_{-1,-1}^i x_{-1,1}^{k-i}\mapsto i \binom{i-1}{j}x_{1,1}^{k-i-1-j} x_{1,-1}^{i-1-j} x_{1,\overline{0}}; \ x_{-1,-1}^i x_{-1,1}^{k-i-1}x_{-1,\overline{0}}\mapsto \binom{i}{j} x_{1,1}^{k-i-1-j} x_{1,-1}^{i-j}.\]
\end{lm}
\begin{proof}
Comparing the weights we see that 
\[x_{-1,-1}^i x_{-1,1}^{k-i}\mapsto C_i x_{1,1}^{k-i-1-j} x_{1,-1}^{i-1-j} x_{1,\overline{0}};  \ x_{-1,-1}^i x_{-1,1}^{k-i-1} x_{-1,\overline{0}}\mapsto D_i x_{1,1}^{k-i-1-j} x_{1,-1}^{i-j}.\]
Without a loss of generality one can assume that $D_j=1$. If $j > i-1$, then the right-hand side of the first map is defined if and only if $C_i=0$. Similarly, if $j > i$, then the right-hand side of the second map is defined if and only if $D_i=0$.

Assume that $j\leq i$. Then
\[\psi_{-k, k-1-2j}({\bf e}^{(i-j)}x_{-1,-1}^i x_{-1,1}^{k-i-1}x_{-1,\overline{0}})=\binom{i}{i-j}\psi_{-k, k-1-2j}(x_{-1,-1}^j x_{-1,1}^{k-j-1}x_{-1,\overline{0}})=\]\[\binom{i}{j}x_{1,1}^{k-1-2j}={\bf e}^{(i-j)}(D_i x_{1,1}^{k-i-1-j} x_{1,-1}^{i-j})=D_i x_{1,1}^{k-1-2j},\]
hence $D_i=\binom{i}{j}$.

Next, assume that $j\leq i-1$. Then
\[\psi_{-k, k-1-2j}({\bf y} x_{-1,-1}^i x_{-1,1}^{k-i})=(k-i)\psi_{-k, k-1-2j}(x_{-1,-1}^i x_{-1,1}^{k-i-1}x_{-1,\overline{0}})=\]\[(k-i)D_i x_{1,1}^{k-i-1-j} x_{1,-1}^{i-j}
=-{\bf y}(C_i x_{1,1}^{k-i-1-j} x_{1,-1}^{i-1-j} x_{1,\overline{0}})=C_i x_{1,1}^{k-i-1-j} x_{1,-1}^{i-j},\]
hence $C_i=-(k-i)D_i=(j-i)\binom{i}{j}\equiv (i-j)\binom{i}{j}\equiv i\binom{i-1}{j}\pmod p$.
\end{proof}

\begin{cor}\label{korko}
Assume $0\leq j<p^{a(k-j)}$, provided $p|k$, or $1\leq j<p^{a(k-j)}$, provided $p\not| k$.

The kernel of $\psi_{-k,k-1-2j}$ has a basis consisting of the following elements:

$x_{-1,-1}^i x_{-1,1}^{k-i}$ such that either $0\leq i\leq j$, $k-j\leq i \leq k$, or $j+1\leq i\leq k-j-1$ and $i\binom{i-1}{j}$ is a multiple of $p$;

$x_{-1,-1}^i x_{-1,1}^{k-i-1}x_{-1,\overline{0}}$ such that either $0\leq i\leq j-1$, $k-j\leq i \leq k-1$, or $j\leq i\leq k-j-1$ and $\binom{i}{j}$ is a multiple of $p$.
\end{cor}
\begin{proof}
If $i\geq k-j$, then either $i=k-j$ and $p|i$, or $i> k-j$ and Kummer's theorem implies $p|\binom{i-1}{j}$. The proof of the second statement is similar.
\end{proof}

\section{Simple composition factors of induced supermodules for $SpO(2|1)$}

\subsection{$SL(2)$-modules}

In this subsection we describe simple composition factors of induced $SL(2)$-modules over the ground field $K$ of positive characteristic $p$.

The induced module $H^0_{SL(2)}(k)$ of the highest weight $k$ has a basis consisting of elements
$x_{1,1}^{k-i}x_{1,-1}^i$ for $i=0, \ldots, k$ and its character $\chi(H^0_{SL(2)}(k))$ equals 
$\sum_{i=0}^k x^{k-2i}$.

The simple module $L_{SL(2)}(k)$ of the highest weight $k$ has a basis consisting of elements
$x_{1,1}^{k-i}x_{1,-1}^i$, where $i=0, \ldots, k$ are such that $\binom{k}{i}$ is not divisible by $p$ and its character
$\chi(L_{SL(2)}(k))$ equals $\sum_{1\leq i\leq k; p\not|\binom{k}{i}} x^{k-2i}$.

Since the multiplicities of all weight spaces involved are either zero or one, in order to find simple composition factors of $H^0_{SL(2)}(k)$, it is enough to exhibit those simple modules $L(\ell_j)$ such that the sum of their characters equals the character of $H^0_{SL(2)}(k)$. Thus this question is purely combinatorial, involving only binomial coefficients, i.e. entries of the Pascal triangle, considered modulo $p$. 

Assume that the $p$-adic expansion of $k+1$ is given as $k+1=\sum_{i=0}^u a_ip^i$, where each $0\leq a_i <p$.

We will consider specific words $w=w_0\ldots w_u$ (to be defined recursively later) of length $u+1$ consisting of symbols $\1$, $\2$, $\3$ and $\4$. To each such word $w$ we assign 
a subset $S(w)$ of the set $\{0, \ldots, k\}$ consisting of those numbers $s$ with the $p$-adic expansion $s=\sum_{i=0}^u s_ip^i$ such that
\[\begin{aligned} 0\leq s_i\leq a_i-1 &\mbox{ if } w_i="\1", \qquad &0\leq s_i\leq a_i &\mbox{ if }w_i="\2", \\
a_i\leq s_i<p &\mbox{ if } w_i="\3" \mbox{ and } & a_i+1\leq s_i <p &\mbox{ if } w_i="\4".
\end{aligned}\]

To each word $w$ we assign a weight $\ell(w)$ by 
\[\ell_k(w)=k-\sum_{w_i="\3"} 2a_ip^i -\sum_{w_i="\4"} 2(a_i+1)p^i.\]

Due to the property $\binom{k}{i}=\binom{k}{k-i}$ we note that the character of a simple module $L_{SL(2)}(\ell)$ is uniquely determined by the set $S(\ell)$ of its nonnegative weights. 
We will see that if $L_{SL(2)}(\ell)$ is a composition factor of $H^0_{SL(2)}(k)$, then $\ell=\ell_k(w)$ for some word $w$ as above.

We will define the ordered sets $W^s=\cup_{j=-1}^{s-1} W_j$ and $W=\cup_{j=-1}^{\infty} W_j$ recursively as follows.

First define $W_{-1}$ to consist of a single word $w="\1\2\ldots \2"$. 
Then define $W_0$ to consist of a single word $w="\3\1\2\ldots \2"$. Assume the ordered sets $W_j=\{w^j_1, \ldots, w^j_{2^j}\}$ for $0\leq j\leq s-1$ are defined, define 
$W_{s}=\{w^{s}_1, \ldots w^{s}_{2^{s}}\}$ as follows.

For $j=1, \ldots 2^{s-1}$, $w^s_j$ is obtained from $w^{s-1}_j$ by replacing its $s$-th symbol $\1$ by $\3$ and its $(s+1)$-st symbol $\2$ by $\1$. The remaining elements
$w^s_j$ for $j=2^{s-1}+1, \ldots, 2^s$ are obtained by taking elements of $W_{-1}$ through $W_{s-2}$ in the order they are listed and changing their $s$-th symbols 
$\2$ to $\4$ and their $(s+1)$-st symbols $\2$ to $\1$.

\begin{ex}
Let $k+1=a_0+a_1p+a_2p^2+a_3p^3+a_4p^4$, where $a_4\neq 0$. Then the ordered set $W^4$ consists of the following words $w$ (listed together with $\ell_k(w)$)

\[\begin{aligned}&\1\2\2\2\2 & \mbox{ of weight  } & k & & & & & \\
&\3\1\2\2\2 & \mbox{ of weight  } & k&-2a_0&&&&\\
&\3\3\1\2\2 & \mbox{ of weight  } & k&-2a_0&-2a_1p&&\\
&\1\4\1\2\2 & \mbox{ of weight  } & k&&-2(a_1+1)p&&\\
&\3\3\3\1\2& \mbox{ of weight  } & k&-2a_0&-2a_1p&-2a_2p^2&\\
&\1\4\3\1\2 & \mbox{ of weight  } & k&&-2(a_1+1)p&-2a_2p^2 & \\
&\1\2\4\1\2 & \mbox{ of weight } & k&&&-2(a_2+1)p^2& \\
&\3\1\4\1\2 & \mbox{ of weight  } & k&-2a_0&&-2(a_2+1)p^2& \\
&\3\3\3\3\1 &  \mbox{ of weight  } & k&-2a_0&-2a_1p&-2a_2p^2&-2a_3p^3   \\
&\1\4\3\3\1 & \mbox{ of weight  } & k&&-2(a_1+1)p&-2a_2p^2&-2a_3p^3 \\
&\1\2\4\3\1 & \mbox{ of weight }  & k&&&-2(a_2+1)p^2&-2a_3p^3 \\
&\3\1\4\3\1 & \mbox{ of weight  } & k&-2a_0&&-2(a_2+1)p^2&-2a_3p^3\\
&\1\2\2\4\1 & \mbox{ of weight  } & k&&&&-2(a_3+1)p^3\\
&\3\1\2\4\1 & \mbox{ of weight  } & k&-2a_0&&&-2(a_3+1)p^3\\
&\3\3\1\4\1 & \mbox{ of weight  } & k&-2a_0&-2a_1p&&-2(a_3+1)p^3\\
&\1\4\1\4\1 & \mbox{ of weight  } & k&&-2(a_1+1)p&&-2(a_3+1)p^3
\end{aligned}\] 
\end{ex}

\begin{pr}\label{simplesl2}
Let $k+1=\sum_{i=0}^u a_ip^i$, where $0\leq a_i <p$.

If $k$ is such that every $a_i\neq 0, p-1$, then the simple composition factors of $H^0_{SL(2)}(k)$ are exactly those $L_{SL(2)}(\ell_k(w))$, where $w\in W^u$ and $\ell_k(w)\geq 0$.

For each $i$ such that $a_i=p-1$, the weights $\ell_k(w)$, corresponding to $w$ such that $w_i="\4"$, are removed from this list. The list of weights after these removals gives all simple
composition factors $L_{SL(2)}(\ell)$ of $H^0_{SL(2)}(k)$. 

For each $i$ such that $a_i=0$, we remove all weights $\ell_k(w)$, corresponding to $w$ such that $w_i="\1"$. After removing duplicate listings of the remaining weights, we obtain the list of all simple composition factors $L_{SL(2)}(\ell)$ of $H^0_{SL(2)}(k)$.
\end{pr}
\begin{proof}
We will be using repeateadly Kummer's theorem that states that $\binom{n}{j}$ is not divisible by $p$ if and only of there is no $p$-adic carry when $n-j$ is added to $j$.

Assume first that every $a_i\neq 0, p-1$.

Using this we derive that the simple module $L_{SL(2)}(k)$ covers exactly those weights $k-2s\geq 0$, where $s=\sum_{i=0}^u s_ip^i$ and $0\leq s_i <p$, for which 
$0\leq s_0\leq a_0-1$ and $0\leq s_i\leq a_i$ for $i>0$.  Thus $S(k)=S(\1\2\ldots \2)$. The largest weight not covered by $S(k)$ is $k-2a_0$.

We will use the notation $b=\sum_{i=0}^u b_ip^i$, where $0\leq b_i <p$.

For the null batch, write $s=a_0+b$.
Since $k-2a_0=(p-1-a_0)+(a_1-1)p + a_2p^2 + \ldots$, the simple module $L_{SL(2)}(k-2a_0)$ covers nonnegative weights $k-2s$ where  $0\leq b_0\leq p-1-a_0$, $0\leq b_1 \leq a_1-1$, 
$0\leq b_2 \leq a_2, \ldots$ which is equivalent to 
$a_0\leq s_0\leq p-1, \, 0\leq s_1 \leq a_1-1, \, 0\leq s_2 \leq a_2 \, \ldots$ and equals $S(\3\1\2\ldots \2)$.
The simples $L_{SL(2)}(k)$ and $L_{SL(2)}(k-2a_0)$ cover all weights of type $k-2s$, where $0\leq s<a_0+a_1p$.

The first batch starts with the weight $k-2a_0-2a_1p$. Write $s=a_0+a_1p+b$. Since
$k-2a_0-2a_1p=(p-1-a_0)+(p-1-a_1)p+(a_2-1)p^2+a_3p^3 + \ldots$, the simple module $L_{SL(2)}(k-2a_0-2a_1p)$ covers nonnegative weights $k-2s$ where  
$0\leq b_0\leq p-1-a_0$, $0\leq b_1 \leq p-1-a_1$, 
$0\leq b_2 \leq a_2-1$, $0\leq b_3 \leq a_3 \ldots$ which is equivalent to 
$a_0\leq s_0\leq p-1, \, a_1\leq s_1 \leq a_1-1, \, 0\leq s_2 \leq a_2-1, \, 0\leq s_3 \leq a_3, \ldots$ and equals $S(\3\3\1\2\ldots \2)$.

The largest missing weight is $k-2(a_1+1)p$ if $a_1\neq p-1$. Write $s=(a_1+1)p+b$. Since
$k-2(a_1+1)p=(a_0-1)+(p-2-a_1)p+(a_2-1)p^2+a_3p^3 + \ldots$, the simple module $L_{SL(2)}(k-2(a_1+1)p)$ covers nonnegative weights $k-2s$ where  
$0\leq b_0\leq a_0-1$, $0\leq b_1 \leq p-2-a_1$, 
$0\leq b_2 \leq a_2-1$, $0\leq b_3 \leq a_3 \ldots$ which is equivalent to 
$0\leq s_0\leq a_0-1, \, a_1+1\leq s_1 \leq p-1, \, 0\leq s_2 \leq a_2-1, \, 0\leq s_3 \leq a_3, \ldots$ and equals $S(\1\4\1\2\ldots \2)$.
All the simple modules listed so far cover all weights of type $k-2s$, where $0\leq s<a_0+a_1p+a_2p^2$.

The simple modules $L_{SL(2)}(k)$ and $L_{SL(2)}(k-2a_0)$ corresponding to $\cup_{j=-1}^0 W_j$ cover all nonnegative weights $k-2s$ such that 
$s_0+s_1p<a_0+a_1p$ and $s_j\leq a_j$ for $j>1$.

The simple modules $L_{SL(2)}(k-2a_0-2a_1p)$ and $L_{SL(2)}(k-2(a_1+1)p)$ corresponding to $W_1$ cover all nonnegative weights $k-2s$ such that
$a_0+a_1p\leq s_0+s_1p+s_2p^2<a_0+a_1p+a_2p^2$ and $s_j\leq a_j$ for $j>2$.

Now we are ready to perform the inductive step. Assume that $k>0$ and we have determined the simple modules $L_{SL(2)}(\ell_k(w))$, where $w\in \cup_{j=-1}^{k-1} W_j$ for $k>0$ as described before and
$S(\ell_k(w))=S(w)$. Also, assume that the simple modules corresponding to $\cup_{j=-1}^{k-2} W_j$ cover all nonnegative weights $k-2s$ such that
$s_0+s_1p+s_2p^2+\ldots+s_{k-1}p^{k-1}<a_0+a_1p+a_2p^2+\ldots+a_{k-1}p^{k-1}$, $0\leq s_k\leq a_k-1$ and $0 \leq s_j\leq a_j$ for $j>k$; and 
the simple modules corresponding to $W_k$ cover all nonnegative weights $k-2s$ such that
$a_0+a_1p+\ldots+a_{k-1}p^{k-1}\leq s_0+s_1p+s_2p^2+\ldots+s_{k-1}p^{k-1}<a_0+a_1p+a_2p^2+\ldots+a_kp^k$, $0 \leq s_k \leq a_k-1$ and $0\leq s_j\leq a_j$ for $j>k$.

The largest weight not covered by simples corresponding to $\cup_{j=-1}^{v-1} W_j$ is $k-2a_0-2a_1p-\ldots -2a_vp^v$. Write $s=(a_0+a_1p+\ldots +a_vp^v)+b$. Since
$k-2(a_0+a_1p+\ldots+a_vp^v)=(p-1-a_0)+(p_1-a_1)p + \ldots (p-1-a_v)p^v +(a_{v+1}-1)+a_{v+2}p^{v+2}+\ldots$, the simple $L_{SL(2)}(k-2(a_0+a_1p+\ldots a_vp^v))$ covers
nonnegative $k-2s$ where
$0\leq b_0\leq p-1-a_0$, \ldots, $0\leq b_v \leq p-1-a_v$, 
$0\leq b_{v+1} \leq a_{v+1}-1$, $0\leq b_{v+2} \leq a_{v+2} \ldots$ which is equivalent to 
$a_0\leq s_0\leq p-1, \ldots, a_v\leq s_v \leq p-1, \, 0\leq s_{v+1} \leq a_{v+1}-1, \, 0\leq s_{v+2} \leq a_{v+2}, \ldots$ and equals $S(\underbrace{\3\ldots \3}_k\1\underbrace{\2\ldots \2}_{u-v})$. Thus the first weight of the $v$-th batch is $k-2a_0-2a_1p-\ldots -2a_vp^v$ and the next weight is $k-2(a_1+1)p-2a_2p^2-\ldots -2a_vp^v$ which covers
nonnegative weights $k-2s$ such that $0\leq s_0\leq a_0=1$, $a_1+1\leq s_1 \leq p-1$, $a_2\leq s_2\leq s_2\leq p-1$, \ldots, $a_v\leq s_v\leq p-1$, $0\leq s_v\leq a_v-1$, 
$0\leq s_{v+1} \leq a_{v+1}$, \ldots, $0\leq s_u\leq a_u$ which equals $S(\1\4\underbrace{\3\ldots \3}_{v-2}\1\underbrace{\2\ldots\2}_{u-v})$.

Using the definition of elements $w^v_1, \ldots, w^v_{2^{v-1}}$ and the inductive assumption, we build consecutive weights $\ell_k(w^v_1)$, \ldots, $\ell_k(w^v_{2^{v-1}})$ such that 
$\cup_{j=1}^{2^{v-1}} S(\ell_k(w^v_j))$ cover all nonnegative weights $k-2s$, where $a_0+a_1p+\ldots a_vp^v\leq s<(a_v+1)p^v$, $0\leq s_{v+1}\leq a_{v+1}-1$ 
and $0\leq s_j\leq a_j$ for $j>v+1$. This confirms that the modules $L_{SL(2)}(\ell_k(w^v_j))$ for $j=1, \ldots, 2^{v-1}$ are composition factors of $H^0_{SL(2)}(k)$.

Afterward, using the definition of elements $w^v_{2^{v-1}+1}, \ldots, w^v_{2^v}$ and the inductive assumption, we build consecutive weights 
$\ell_k(w^v_{2^{v-1}+1})$, \ldots, $\ell_k(w^v_{2^v})$ such that 
$\cup_{j=2^{v-1}+1}^{2^v} S(\ell_k(w^v_j))$ cover all nonnegative weights $k-2s$, where $(a_v+1)p^v\leq s<a_0+a_1p+\ldots +a_vp^v$, $0\leq s_{v+1}\leq a_{v+1}-1$ 
and $0\leq s_j\leq a_j$ for $j>v+1$. This confirms that the modules $L_{SL(2)}(\ell_k(w^v_j))$ for $j=2^{v-1}+1, \ldots, 2^v$ are composition factors of $H^0_{SL(2)}(k)$.

After $u-1$ batches we cover all nonnegative weights $k-2s$ where $0\leq s\leq a_0+\ldots+a_up^u=k+1$, hence all nonnegative weights $k-2s$ for $0\leq s$. This shows that we have found all simple composition factors of $H^0_{SL(2)}(k)$ and they are described in the statement of the proposition.

If $a_i=p-1$, then we omit the terms $w$ with $w_i="\4"$ because the corresponding set $S(w)$ is empty. 

Analogously, of $a_i=0$, then we need to omit the terms $w$ with $w_i="\1"$because the corresponding set $S(w)$ is empty. Also, in this case some of the weights on our list can appear more than once. In this case the correct description of the set $S(w)$ is given by the latest appearance of $w'$ such that $\ell_k(w)=\ell_k(w')$. By eliminating all $w$ with $w_i="\1"$ and 
duplicate entries on our list we arrive at the list of all simple composition factors of $H^0_{SL(2)}(k)$.

In both cases of appearances of $a_i=p-1$ or $a_i=0$ the previous inductive arguments carry over.
\end{proof}

\begin{rem}\label{linksl2}
As a direct consequence of Proposition \ref{simplesl2} one can derive the well-known linkage principle for $SL(2)$ over a ground field $K$ of characteristic $p$.
Denote by $d(l)$ the defect of the weight $l$, that is the largest integer such that $l+1\equiv 0 \pmod{p^d}$.
Then the nonnegative weights $l$ and $k$ are $SL(2)$-linked if and only if $d(l)=d(k)$ and either 
$l\equiv k \pmod{2p^{d+1}}$ or $l\equiv -k-2 \pmod{2p^{d+1}}$.
The way the defect $d(k)$ appears in the above consideration is that simple composition factors of $H^0_{SL(2)}(k)$ correspond to words $w\in W(k)$ starting with 
$\underbrace{\3\ldots \3}_{d(k)}$. This implies that linked $SL(2)$-weights must have the same defect. One can also simplify the listing of simple composition factors for $H^0_{SL(2)}(k)$
by replacing every appearance of $a_ip^i$ and $(a_i+1)p^i$ by $a_{i+d(k)}p^{i+d(k)}$ and $(a_{i+d(k)}+1)p^{i+d(k)}$, respectively (since all $a_0=\ldots =a_{d-1}=0$.) 
\end{rem}

\subsection{SpO(2$|$1)-supermodules} 

In this subsection we describe simple composition factors of induced $SpO(2|1)$-supermodules over the ground field $K$ of positive characteristic $p$.

\begin{lm}\label{osptosl2}
The induced $SpO(2|1)$-supermodule $H^0_-(k)$, considered as an $SL(2)$-module, is isomorphic to $H^0_{SL(2)}(k)$ if $p|k$ and to $H^0_{SL(2)}(k)\oplus H^0_{SL(2)}(k-1)$ 
if $p\not | k$.
The simple $SpO(2|1)$-supermodule $L_-(k)$, considered as an $SL(2)$-module, is isomorphic to $L_{SL(2)}(k)$ if $p|k$ and to $L_{SL(2)}(k)\oplus L_{SL(2)}(k-1)$ if $p\not | k$.
\end{lm}
\begin{proof}
It suffices to compare the characters of the supermodules $H^0_-(k)$, $H^0_{SL(2)}(k)$ and $H^0_{SL(2}(k-1)$, and 
$L_-(k)$, $L_{SL(2)}(k)$ and $L_{SL(2)}(k-1)$, respectively, and to take into account that the weights $k$ and $k-1$ are not $SL(2)$-linked
(which is clear but also follows from Remark \ref{linksl2}).
\end{proof}

We need to introduce further notation related to Proposition \ref{simplesl2}. Consider a word $w\in W \setminus W_{-1}$ and the corresponding weight $\ell_k(w)$. 
We say that $w$ and $\ell_k(w)$ are of the {\it first kind} if $w_0="\3"$ (this is equivalent to $\ell_k(w)\equiv k-2a_0 \pmod{2p}$). 
We say that $w$ and $\ell_k(w)$ are of the {\it second kind} if $w_0="\1"$ (this is equivalent to $a_0\neq 0$ and $\ell_k(w)\equiv k\pmod{2p}$).  

\begin{pr}\label{simplesosp}
Let $k>0$. If $k \not \equiv 0,-1 \pmod p$, then the simple composition factors of $H^0_-(k)$ are $L_-(k)$, all $L_-(\ell_{k-1}(w))$ such that $w\in W$ is of the first kind and 
$\ell_{k-1}(w)\geq 0$, and 
all $L_-(\ell_k(w))$ such that $w\in W$ is of the second kind and $\ell_k(w)\geq 0$. 

If $k \equiv -1 \pmod p$, then the simple composition factors of $H^0_-(k)$ are $L_-(k)$, all $L_-(\ell_{k-1}(w))$ such that $w\in W$ is of the first kind and 
$\ell_{k-1}(w)\geq 0$, and 
all $L_-(\ell_k(w))$ such that $w\in W$ is of the first kind and $\ell_k(w)\geq 0$. 

If $k\equiv 0 \pmod p$, then the simple composition factors of $H^0_-(k)$ are $L_-(k)$, $L_-(k-1)$, all $L_-(\ell_{k-1}(w))$ such that $w\in W$ is of the first kind 
and $\ell_{k-1}(w)\geq 0$, and 
all $L_-(\ell_k(w))$ such that $w\in W$ is of the second kind and $\ell_k(w)\geq 0$. 
\end{pr}
\begin{proof}
The proof follows from Proposition \ref{simplesl2} and Lemma \ref{osptosl2}. 

Assume that $p$ does not divide $k$ and consider all appearing supermodules as $SL(2)$-modules. Then $H^0_-(k)=H^0_{SL(2)}(k)\oplus H^0_{SL(2)}(k-1)$ and 
$L_-(k)=L_{SL(2)}(k)\oplus L_{SL(2)}(k-1)$. 

Assume further that $k\not \equiv -1 \pmod p$. Then $a_0>0$, $k=(k-1)+1=(a_0-1)+ a_1p+ \ldots + a_sp^s$ and $(k-1)-2(a_0-1)=k+1-2a_0\not\equiv 0 \pmod p$ shows that the two simple modules $L_{SL(2)}(\ell_k(w))$ and $L_{SL(2)}(\ell_{k-1}(w))$ for $w\in W$ of the first kind combine to form $L_-(\ell_{k-1}(w))$ corresponding to this $w\in W$. 
Therefore, all simple modules described in the first statement appear as composition factors of $H^0_-(k)$ and there are no other composition factors.

Now assume $k\equiv -1 \pmod p$. Then $d(k)>0$ and 
\begin{equation}\label{stein}
k=(p-1)+ (p-1)p+\ldots (p-1)p^{d(k)-1}+ (a_{d(k)}-1)p^{d(k)}+\ldots +a_sp^s.
\end{equation}

In this case there are no simple composition factors of $H^0_-(k)$ corresponding to $w\in W$ of the second kind. Consider $w\in W$ of the first kind. Then 
$(k-1)-2(p-1)\equiv 0 \pmod p$ and all simple composition factors $L_{SL(2)}(\ell_{k-1}(w))$ of $H^0_-(k)$ for $w\in W$ of the first kind are simple $SpO(2|1)$-supermodules. 
Using (\ref{stein}) we infer that for each simple $L_{SL(2)}(\ell_k(w_1))$, where $w_1\in W$ of the first kind, appearing as composition factor of $H^0_-(k)$, the simple
$L_{SL(2)}(\ell_k(w_1)-1)$ also appears as a composition factor $L_{SL(2)}(\ell_{k-1}(w_2))$ of $H^0_-(k)$ for appropriate $w_2\in W$ of the second kind. 
We note that all weights $\ell_k(w_1)$, where $w_1\in W$ is of the first kind, have $SL(2)$-defect equaled to $d(k)$.
The correspondence between words $w_1$ and $w_2$ is as follows. The entries at the positions $d(k)+1$ through $s$ in $w_1$ and $w_2$ are the same. 
The first entries of $w_1$ and $w_2$ are $\underbrace{\3\ldots \3}_{d(k)}w_{1,d(k)}$ and $\1\underbrace{\2\ldots \2}_{d(k)-1}w_{2,d(k)}$, respectively. 
For $w_{1,d(k)}="\1"$ the corresponding entry $w_{2,d(k)}="\2"$ and for $w_{1,d(k)}="\3"$ the corresponding entry $w_{2,d(k)}="\4"$.
Since $L_{SL(2)}(\ell_k(w_1))$ and $L_{SL(2)}(\ell_k(w_1)-1)$, for $w_1\in W$ of the first kind, combine to form $L_-(\ell_k(w_1))$, the second claim follows. 

Assume now that $k\equiv 0 \pmod p$. It is clear that $L_-(k-1)$ is a composition factor of $H^0_-(k)$.
Since $a_0=1$, we have $k=(k-1)+1=0+ a_1p+ \ldots + a_sp^s$ and $(k-1)-2(a_0-1)=k+1-2a_0\not\equiv 0 \pmod p$ shows that 
the simple supermodules $L_{SL(2)}(\ell_k(w))$ and $L_{SL(2)}(\ell_{k-1}(w))$ for $w\in W$ of the first kind combine to form $L_-(\ell_{k-1}(w))$ corresponding to this $w\in W$. 
There are no composition factors corresponding to $\ell_{k-1}(w)$ for $w\in W$ of the second kind because $a_0-1=0$. 
All composition factors corresponding to $w\in W$ of second kind satisfy 
$\ell_k(w)\equiv 0 \pmod p$. Therefore, such simple $L_{SL(2)}(\ell_k(w))$ are simple $SpO(2|1)$-supermodules. 
The claim follows.
\end{proof}

\begin{rem}
We would like to direct reader's attention to the fact that if $p|k$, then the weights $k$ and $k-1$ are $SpO(2|1)$-linked. This is the reason why the defect of weights known 
from $SL(2)$-linkage does not play a role in the $SpO(2|1)$-linkage.
\end{rem}

For $0\leq k <p$ denote by $\mathscr{B}_k$ the set of all nonnegative integers $l$ such that $l\equiv k \pmod{2p}$ or $l\equiv 2p-1-k \pmod{2p}$. 

\begin{pr}\label{blocksosp}
The sets $\mathscr{B}_0, \ldots \mathscr{B}_{p-1}$ are blocks of the category of $SpO(2|1)$-supermodules.
\end{pr}
\begin{proof}
It is easy to verify that all weights of simple composition factors of the induced $SpO(2|1)$-supermodule $H^0_-(k)$ belong to the set $\mathscr{B}_{a_0}$. On the other hand, for $t\geq 0$
the weights $a_0+2pt$ and $k=2p-1-a_0+2pt$ are linked because the simples $L_-(a_0+2pt)$ and $L_-(2p-1-a_0+2pt)$ are composition factors of $H_-^0(2p-1-a_0+2pt)$ according to Proposition \ref{simplesosp}.
Indeed, if $a_0=p-1$, then $k$ is divisible by $p$ and $a_0+2pt=k-1$; otherwise $a_0+2pt=\ell_{k-1}(w)$ for the single element $w\in W_0$. 

The weights $k=2p+a_0+2pt$ and $2p-1-a_0+2pt$ are linked because the simples $L_-(2p-1-a_0+2pt)$ and $L_-(2p+a_0+2pt)$ are composition factors of $H_-^0(2p+a_0+2pt)$ according to 
Proposition \ref{simplesosp}. 
Indeed, if $a_0=0$, then $k$ is divisible by $p$ and $2p-1-a_0+2pt=k-1$; otherwise $2p-1-a_0+2pt=\ell_{k-1}(w)$ for the single element $w\in W_0$. 
\end{proof}

\begin{pr}\label{kercokk-1}
Let $p$ divides $k$. Then the kernel of $\psi_{-k,k-1}$ has simple composition factors $L_-(k)$ and $L_-(\ell_k(w))$, where $w\in W$ is of the second kind,
and the cokernel of $\psi_{-k,k-1}$ has simple composition factors $L_-(\ell_{k-2}(w))$, 
where $w\in W$ is of the first kind.

Consequently, the image of $\psi_{-k,k-1}$ has simple composition factors $L_-(k-1)$ and $L_-(\ell_{k-1}(w))$, where $w\in W$ is of the first kind.
\end{pr}
\begin{proof}
According to Corollary \ref{korko}, the kernel of $\psi_{-k,k-1}$ has weights $k-2i$, where $0\leq i \leq k$ is such that $i$ is a multiple of $p$. 
Write $k=(a_0-1)+\ldots +a_sp^s$ as before and $i=i_0+i_1p+ \ldots +i_sp^s$, where $0\leq i_j<p$. Then indices $i$ divisible by $p$ are characterized by $i_0\leq a_0-1$ and 
$i_1, \ldots i_s$ arbitrary. Therefore the kernel of $\psi_{-k,k-1}$ consists of $L_-(\ell_k(w))$, where $w\in W$ satisfies $w_0="\1"$.

Also, according to Corollary \ref{korko}, the cokernel of $\psi_{-k,k-1}$ has weights $k-2-2(i-1)$, where $0< i < k$ is such that $i$ is a multiple of $p$. 
Let $k-2=(b_0-1)+b_1p+\ldots +b_sp^s$ and $k=k_0+k_1p+\ldots+ k_sp^s$ be $p$-adic expansions of $k-2$ and $k=i-1$.
Then we can characterize the weights $k-2-2k$ as those weights in the module $H^0_{SL(2)}(k-2)$ satisfying $k_0\geq b_0$ and $k_1, \ldots, k_s$ are arbitrary.
Therefore, the cokernel of $\psi_{-k,k-1}$ consists of $L_{SL(2)}(\ell_{k-2}(w))$, where $w\in W$ satisfies $w_0="\3"$. Since for every $w\in W$ such that $w_0="\3"$ we have
$\ell_{k-2}(w)$ divisible by $p$, we derive that $L_{SL(2)}(\ell_{k-2}(w))=L_-(\ell_{k-2}(w))$ and conclude the proof of the first part of our claim.
The second part of the claim follows from Proposition \ref{simplesosp}.
\end{proof}

Note that the set of weights $\ell_{k-2}(w)$, where $w\in W$ is of the first kind is the same as the set of weights $\ell_{k-2p}(w)$, where $w\in W$ is of the second kind.

\begin{pr}\label{kercokk-1-2j}
Let $0<j<p^{a(k-j)}$, $j=(a_0-1)+\ldots + a_{t-1}p^{t-1}$. Then the image of $\psi_{-k,k-1-2j}$ has simple composition factors $L_-(\ell_{k-1}(w))$, where $w\in W$ 
satisfies $w_0\ldots w_{t-1}="\3\ldots \3"$.
These simple composition factors are also described as $L_-(\ell_{k-1-2j}(w))$, where $w\in W$ satisfies $w_0\ldots w_{t-1}="\1\underbrace{\2\ldots \2}_{t-1}"$.
\end{pr}
\begin{proof}
Using Lemma \ref{towardlinkage+} we determine that the image of $\psi_{-k,k-1-2j}$ has weights of two types. First type is of the form $k-2i$, where $j+1\leq i\leq k-1-j$ are such that 
$i\binom{i-1}{j}$ is not divisible by $p$ and the second type is of the form $k-1-2i$, where $j\leq i\leq k-1-j$ are such that $\binom{i}{j}$ is not divisible by $p$.

Assume first that $j\not\equiv p-1\pmod p$, which implies $a_0>0$.

Let $k=(a_0-1)+a_1p+\ldots +a_sp^s$, $j=(a_0-1)+a_1p+\ldots +a_{t-1}p^{t-1}$, where $t\leq a(k-j)$, and $i=i_0+i_1p+\ldots + i_sp^2$ be $p$-adic expansions of $m$, $j$ and $i$.

Since $\binom{i}{j}$ is not divisible by $p$ if and only if there is no $p$-adic carry when $j$ is added to $i-j$, this happens if and only if 
$i_0\geq a_0-1, i_1\geq a_1, \ldots, i_{t-1}\geq a_{t-1}$. Therefore, the set of weights of the second type correspond to simple supermodules $L_{SL(2)}(\ell_{k-1}(w))$, where 
$w\in W$ satisfies
$w_0\ldots w_{t-1}="\3\ldots \3"$.

Since $i\binom{i-1}{j}=(j+1)\binom{i}{j+1}$ and $j+1$ is not divisible by $p$, we derive that $i\binom{i-1}{j}$ is not divisible by $p$, if and only if, 
there is no $p$-adic carry when $j+1$ is added to $i-j-1$. Since $j+1=a_0+a_1p+\ldots+a_{t-1}p^{t-1}$, this happens if and only if $i_0\geq a_0, \ldots, i_{t-1}\geq a_{t-1}$.

Therefore the set of weights of the first type correspond to simple supermodules $L_{SL(2)}(\ell_k(w))$, where $w\in W$ satisfies $w_0\ldots w_{t-1}="\3\ldots \3"$.

Combining the simples $L_{SL(2)}(\ell_k(w))$ and $L_{SL(2)}(\ell_{k-1}(w))$ for $w\in W$ 
such that $w_0\ldots w_{t-1}="\3\ldots \3"$ we obtain $L_-(\ell_{k-1}(w))$, where $w\in W$, proving the first claim in the case $j\not\equiv p-1\pmod p$.

Now assume $j\equiv p-1 \pmod p$. 

Let $k=p-1+b_1p+\ldots +b_sp^s$, $j=p-1+b_1p+\ldots +b_{t-1}p^{t-1}$, where $t\leq a(k-j)$, and $i=i_0+i_1p+\ldots + i_sp^2$ be $p$-adic expansions of $m$, $j$ and $i$.
The set of weights of the second type correspond to simple supermodules $L_{SL(2)}(\ell_{k-1}(w))$, where $w\in W$ satisfies
$w_0\ldots w_{t-1}="\3\ldots \3"$. The set of weights of the first type is empty since $i\binom{i-1}{j}=(j+1)\binom{i}{j+1}$ is divisible by $p$.
In this case the weights $\ell_{k-1}(w)$, where $w\in W$ is such that $w_0\ldots w_{t-1}="\3\ldots \3"$, are divisible by $p$ because $k-1-2(p-1)\equiv 0 \pmod p$. This implies
that $L_-(\ell_{k-1}(w))=L_{SL(2)}(\ell_{k-1}(w))$ for such $w$ and concludes the proof of the first part of our claim.

To show the second claim it is enough to observe that 
the weight $\ell_{k-1}(w)$ for a word $w\in W$ such that $w="\underbrace{\3\ldots \3}_tw_t\ldots w_{s}"$ equals the weight $\ell_{k-1-2j}(z)$ for a word  
$z\in W$ such that $z="\1\underbrace{\2\ldots \2}_{t-1}z_t\ldots z_{s}"$, where $z_k=w_k$ for $k>t$, $z_t="\2"$ if $w_t="\1"$ and $z_t="\4"$ if $w_t="\3"$.
\end{proof}

\section{Representations of $G_r T$ for $G=SpO(2|1)$}
Throughout this section $G=SpO(2|1)$. 

\subsection{Induced and simple supermodules}

An element $g\in G$ belongs to $G_r$ if and only if $g_{11}^{p^r}=g_{-1, -1}^{p^r}=1$ and $g_{-1, 1}^{p^r}=g_{1, -1}^{p^r}=0$. 
Consequently, $g\in G$ belongs to $G_r T$ if and only if $g_{-1, 1}^{p^r}=g_{1, -1}^{p^r}=0$. Indeed, the necessary condition is obvious. Conversely, if 
$g_{-1, 1}^{p^r}=g_{1, -1}^{p^r}=0$, then 
$1=(g_{11}g_{-1, -1}-g_{1, -1}g_{-1, 1}+g_{1, \overline{0}}g_{-1, \overline{0}})^{p^r}=(g_{11}g_{-1, -1})^{p^r}$ 
implies that $g_{11}$ is invertible. Thus
\[\left(\begin{array}{ccc}
g_{11} & g_{1, -1} & g_{1, \overline{0}} \\
g_{-1, 1} & g_{-1, -1} & g_{-1, \overline{0}} \\
g_{\overline{0}, 1} & g_{\overline{0}, -1} & g_{\overline{0}, \overline{0}}
\end{array}\right)=\left(\begin{array}{ccc}
1 & g_{1, -1}g_{11} & g_{1, \overline{0}} \\
g_{-1, 1}g_{11}^{-1} & g_{-1, -1}g_{11} & g_{-1, \overline{0}} \\
g_{\overline{0}, 1}g_{11}^{-1} & g_{\overline{0}, -1}g_{11} & g_{\overline{0}, \overline{0}}
\end{array}\right)\left(\begin{array}{ccc}
g_{11} & 0 & 0 \\
0 & g_{11}^{-1} & 0 \\
0 & 0 & 1
\end{array}\right),\]
hence $g\in G_r T$. 

The following lemma is now evident.
\begin{lm}\label{K[G_r T]}
The superalgebra $K[G_r T]$ is generated by the elements $x_{i j}$ for $i, j\in\{1, -1\}$, and $x_{1, \overline{0}}, x_{-1, \overline{0}}$ that are subject to the relations
\[x_{11}x_{-1, -1}-x_{1, -1}x_{-1, 1}+
x_{1, \overline{0}}x_{-1, \overline{0}}=1, x_{1, -1}^{p^r}=x_{-1, 1}^{p^r}=0.\]
\end{lm}

The isomorphism of superschemes $\phi : G_r T\to B_r^- T\times U^+_r$ is given by
\[\left(\begin{array}{ccc}
g_{11} & g_{1, -1} & g_{1, \overline{0}} \\
g_{-1, 1} & g_{-1, -1} & g_{-1, \overline{0}} \\
g_{\overline{0}, 1} & g_{\overline{0}, -1} & g_{\overline{0}, \overline{0}}
\end{array}\right)\mapsto\]\[\left(\begin{array}{ccc}
g_{11} & 0 & 0 \\
g_{-1, 1} & g_{11}^{-1} & g_{-1, \overline{0}}-g_{-1, 1}g_{11}^{-1}g_{1, \overline{0}}\\
g_{-1, 1}g_{1, \overline{0}}-g_{11}g_{-1, \overline{0}} & 0 & 1\end{array}\right)\times
\left(\begin{array}{ccc}
1 & g_{11}^{-1}g_{1, -1} & g_{11}^{-1}g_{1, \overline{0}}\\
0 & 1 & 0 \\
0 & g_{11}^{-1}g_{1, \overline{0}} & 1\end{array}\right).\]

It follows from the above and from Lemma \ref{K[G_r T]} that the superalgebra $K[B_r^- T]$ is generated by elements
$y_{11}^{\pm 1}, y_{-1, 1}$ and $y_{-1, \overline{0}}$ subject to the relation $y_{-1, 1}^{p^r}=0$.
Similarly, the superalgebra $K[U^+_r]$ is generated by elements $u_{1, -1}, u_{1, \overline{0}}$ subject to the relation $u_{1, -1}^{p^{r}}=0$.

The dual superalgebra morphism $\phi^* : K[B^-_r T]\otimes K[U^+_r]\to K[G_r T]$ is defined on the generators as
\[\phi^*(y_{11})=x_{11}, \phi^*(y_{-1, 1})=x_{-1, 1}, \phi^*(y_{-1, \overline{0}})=x_{-1, \overline{0}}-x_{-1, 1}x_{11}^{-1}x_{1, \overline{0}},\]
\[\phi^*(u_{1, -1})=x_{11}^{-1}x_{1, -1} \mbox{ and } \phi^*(u_{1, \overline{0}})=x_{11}^{-1}x_{1, \overline{0}}.\]
\begin{lm}\label{Frobsimple-}
The induced supermodule $H^0_{r, -}(l)=ind^{G_r T}_{B_r^- T} K_l$ has a basis
\[x_{11}^{l-k-\epsilon}x_{1, -1}^k x_{1, \overline{0}}^{\epsilon} \mbox{ for } 0\leq k < p^r \mbox{ and } \epsilon=0, 1.\]
Moreover, if $p|l$, then the socle $L_{r, -}(l)$ of $H^0_{r, -}(l)$ has a basis
\[x_{11}^{l-k}x_{1, -1}^k \mbox{ for } 0\leq k < p^r \mbox{ such that } p \not | \binom{l}{k};\]
if $p\not |l$, then it has a basis
\[x_{11}^{l-k}x_{1, -1}^k \mbox{ for } 0\leq k < p^r \mbox{ such that } p\not |\binom{l}{k}, \mbox{ and } \]
\[x_{11}^{l-1-t}x_{1, -1}^t x_{1, \overline{0}} \mbox{ for } 0\leq t< p^r \mbox{ such that } p\not |\binom{l-1}{t}.\]
\end{lm}
\begin{proof}
By Lemma 8.2 of \cite{zub1} the elements $\phi^*(y_{11}^l f)$, when $f$ runs over a basis of $K[U_r^+]$, form a basis of $H^0_{r, -}(l)$. The proof of remaining statements 
follows from formulas in the proof of Lemma \ref{soclewrtlower}.
\end{proof}
The proof of the next lemma is similar to the proof of Lemma \ref{Frobsimple-}.
\begin{lm}\label{Frobsimple+}
The induced supermodule $H^0_{r, +}(-l)=ind^{G_r T}_{B_r^+ T} K_{-l}$ has a basis
\[x_{-1, -1}^{l-k-\epsilon}x_{-1, 1}^k x_{-1, \overline{0}}^{\epsilon} \mbox{ for } 0\leq k < p^r \mbox{ and } \epsilon=0, 1.\]
If $p|l$, then the socle $L_{r, +}(-l)$ of $H^0_{r, +}(-l)$ has a basis
\[x_{-1, -1}^{l-k}x_{-1, 1}^k \mbox{ for } 0\leq k < p^r \mbox{ such that } p\not |\binom{l}{k};\]
if $p\not |l$, then $L_{r, +}(-l)$ has a basis
\[x_{-1, -1}^{l-k}x_{-1, 1}^k \mbox{ for } 0\leq k< p^r  \mbox{ such that } p\not |\binom{l}{k}, \mbox{ and }\]
\[x_{-1, -1}^{l-1-t}x_{-1, 1}^t x_{-1, \overline{0}} \mbox{ for } 0\leq t< p^r \mbox{ such that }  p\not | \binom{l-1}{t}.\]\end{lm}

\begin{lm}
The superspace $Hom_G(H_{r, +}(-k), H_{r, -}(l))\neq 0$ if and only if $l=2p^r-k-1$.
\end{lm}
\begin{proof}
As in the proof of Lemma \ref{nontrivialmorphism} we use the following formulas
\[{\bf f}^{(t)}x_{-1, -1}^{k-\epsilon-s-t}x_{-1, 1}^{s+t} x_{-1, \overline{0}}^{\epsilon} =
\left(\begin{array}{c} s+t \\ t\end{array}\right)x_{-1, -1}^{k-\epsilon-s}x_{-1, 1}^{s} x_{-1, \overline{0}}^{\epsilon} \mbox{ and }\]
\[{\bf y}x_{-1, -1}^{k-\epsilon-k}x_{-1, 1}^k x_{-1, \overline{0}}^{\epsilon} =k x_{-1, -1}^{k-\epsilon-k}x_{-1, 1}^{k-1} x_{-1, \overline{0}}^{\epsilon+1} +\epsilon x_{-1, -1}^{k-\epsilon-k+1}x_{-1, 1}^k\]
for $0\leq s< p^r$ and $0< t< p^r-s$. 
The second formula shows that every $x_{-1, -1}^{k-k}x_{-1, 1}^k$ belongs to $rad_{B_r^- T} H_{r, +}(-k)=Dist(U^-_r)H_{r, +}(-k)$. 

We will show that for $0\leq s < p^r-1$ there is $0< t<p^r -s$ such that $p\not|\binom{t+s}{t}$. If this is not the case, 
Lemma \ref{Fero's responsibility} applied to $k=p^r$ and $j=p^r-s-1$ implies $0\leq p^r-s-1<p^{a(s+1)}$. Since $p^r-s-1$ is divisible by $p^{a(s+1)}$,
 this is only possible if $p^r-s-1=0$, which is a contradiction.

Therefore every $x_{-1, -1}^{k-1-s}x_{-1, 1}^s x_{-1, \overline{0}}$, where $0\leq s < p^r-1$, also belongs to $rad_{B_r^- T} H_{r, +}(-k)$.
Thus $H_{r, +}(-k)/rad_{B_r^- T} H_{r, +}(-k)$ is spanned by the single element $x_{-1, -1}^{k-p^r}x_{-1, 1}^{p^r-1} x_{-1, \overline{0}}$ which means 
that $Hom_G(H_{r, +}(-k), H_{r, -}(l))\neq 0$ if and only if $l=2p^r-1-k$.
\end{proof}
Let $H$ be an algebraic supergroup and $R$ be its subsupergroup. Let $W=Kw$ be a one-dimensional $R$-supermodule and $V$ be a $H$-supermodule. 
For $\psi\in Hom_R(V, W)$ fix a (homogeneous) basis of $V$, say $v_1, \ldots , v_k$ such that $\psi(v_1)=w$ and $\psi(v_i)=0$ for $2\leq i\leq k$. 
Denote by $\widetilde{\psi} : V\to ind^H_R W$ the unique extention of $\psi$ given by the Frobenius reciprocity.
\begin{lm}\label{inducedmorphisms}
If we identify $ind^H_R W$ with a supersubmodule of $K[H]$ as in Lemma \ref{induced}, then $\widetilde{\psi}(v)=f_1$, where $\tau_V : v\mapsto\sum_{1\leq i\leq k}v_i\otimes f_i$ is a supercomodule map of $V$.
\end{lm}
\begin{proof}
The isomorphism $Hom_R(V, W)\to Hom(V, ind^H_R W)$ is given by $\psi\mapsto (\psi\otimes id_{K[H]})\tau_V$ for $\psi\in Hom_R(V, W)$.
\end{proof}

Assume that $k=tp^r+s$, where $t\leq 0$ and $0\leq s<p^r$ and define $\tilde{k}=k-(t-1)p^r$. 

\begin{lm}\label{morphr}
Let $\psi_{r, k}: H_{r, +}(-k)\to H^0_{r, -}(2p^r-k-1))$ be a nonzero morphism such that
$\psi_{r, k}(x_{-1,-1}^{k-p^r}x_{-1,1}^{p^r-1}x_{-1,\overline{0}})=x_{1,1}^{2p^r-k-1}$.
 
If $p^r\leq k < 2p^r$, then 
\[\psi_{r, k}(x_{-1,-1}^{k-i}x_{-1,1}^i)=(k-i)\binom{k-i-1}{k-p^r} x_{1,1}^{i-k+p^r-1} x_{1,-1}^{p^r-i-1} x_{1,\overline{0}},\]
and
\[\psi_{r, k}(x_{-1,-1}^{k-i-1}x_{-1,1}^i x_{-1,\overline{0}})=\binom{k-i-1}{k-p^r} x_{1,1}^{i-k+p^r} x_{1,-1}^{p^r-i-1}\]
for every $0\leq i <p^r$.

If $k\not\in [p^r, 2p^r-1]$, then
\[\psi_{r, k}(x_{-1,-1}^{k-i}x_{-1,1}^i)=(\tilde{k}-i)\binom{\tilde{k}-i-1}{\tilde{k}-p^r} x_{1,1}^{i-k+p^r-1} x_{1,-1}^{p^r-i-1} x_{1,\overline{0}}\]
and
\[\psi_{r, k}(x_{-1,-1}^{k-i-1}x_{-1,1}^i x_{-1,\overline{0}})=\binom{\tilde{k}-i-1}{\tilde{k}-p^r} x_{1,1}^{i-k+p^r} x_{1,-1}^{p^r-i-1}\]
for every $0\leq i <p^r$.
\end{lm}
\begin{proof}
According to Lemma \ref{inducedmorphisms}, the morphism $\psi_{-k, k-1-2j}$ is defined as follows. 
If $v\in H^0_+(-k)$ is a basis vector and the summand $x_{-1,-1}^j x_{-1,1}^{k-j-1}x_{-1,\overline{0}}\otimes x$ appears in the expression for $\Delta(v)$, then 
$\psi_{-k, k-1-2j}(v)=x$.

Assume now that $p^r\leq k < 2p^r$ and set $j=k-p^r$. Then for every basis vector $v$ of $H^0_{r, +}(-k)$, $\psi_{r, k}(v)$ can be defined as the image of $\psi_{-k, k-1-2j}(v)$ with respect to the superalgebra epimorphism $K[G]\to K[G_r T]$. Thus the first statement follows by Lemma \ref{towardlinkage+}.

The second statement can be proven by the following observation. Since the element $x_{-1,-1}^{p^r}$ is group-like, Lemma \ref{inducedmorphisms} implies that $\psi_{r, k}(v x_{-1,-1}^{p^r (t-1)})=\psi_{r, \tilde{k}}(v) x_{1,1}^{-p^r(t-1)}$ for every $v\in H^0_{r, +}(\tilde{k})$.
\end{proof}

The following lemma is an analogue of Lemma \ref{osptosl2}.
Denote by $G_{ev,r}$ the $r$-th Frobenius kernel of $G_{ev}$, by $H^0_{ev,r,-}(k)$ and $L_{ev,r,-}(k)$ the $G_{ev,r}T$-induced and $G_{ev,r}T$-simple supermodule, respectively, of the highest weight $k$. 

\begin{lm}\label{osptosl2r}
The induced $G_rT$-supermodule $H^0_{r, -}(k)$, considered as an $G_{ev,r}T$-supermodule, is isomorphic to $H^0_{ev,r,-}(k)$ if $p|k$ and to $H^0_{ev,r,-}(k)\oplus H^0_{ev,r,-}(k-1)$ if $p\not|k$.
The simple $G_rT$-supermodule $L_{r,-}(k)$, considered as an $G_{ev,r}T$- supermodule, is isomorphic to $L_{ev,r,-}(k)$ if $p|k$ and to $L_{ev,r,-}(k)\oplus L_{ev,r,-}(k-1)$ if 
$p\not|k$.
\end{lm}
\begin{proof}
It suffices to compare the characters of the supermodules $H^0_{r, -}(k)$, $H^0_{ev,r,-}(k)$ and $H^0_{ev,r,-}(k-1)$, and 
$L_{r,-}(k)$, $L_{ev,r,-}(k)$ and $L_{ev,r,-}(k-1)$, respectively, and to take into account that the weights $k$ and $k-1$ are not $G_{ev,r}T$-linked.
\end{proof}

\begin{rem}
Since $x_{1,1}^{p^r}$ and $x_{-1,-1}^{p^r}$ are group-like elements in $G_rT$, the induced and simple $G_rT$-supermodule of the highest weight $k$ can be expressed as:
\[H^0_{r, -}(k) \simeq H^0_{r, -}(k+tp^r) \otimes K_{-tp^r} \text{ and } L_{r,-}(k) \simeq L_{r,-}(k+tp^r) \otimes K_{-tp^r}.\]
Therefore it is enough to describe the structure of induced and simple $G_rT$-supermodules of nonnegative highest weight $k$, or even better for $p^r\leq k <2p^r$, and the general case is then obtained using a shift by an appropriate multiple of $p^r$. 
\end{rem}

\subsection{Simple composition factors of induced supermodules}

Recall that $W^r=\cup_{j=-1}^{r-1} W_j$. The following proposition is an analogue of Proposition \ref{simplesosp}. 

\begin{pr}\label{simplesospr}
If $l \not \equiv 0,-1 \pmod p$, then the simple composition factors of $H^0_{r, -}(l)$ are $L_{r, -}(l)$, all $L_{r, -}(\ell_{l-1}(w))$ such that $w\in W^r$ is of the first kind, and 
all $L_{r, -}(\ell_l(w))$ such that $w\in W^r$ is of the second kind. 

If $l \equiv -1 \pmod p$, then the simple composition factors of $H^0_{r, -}(l)$ are $L_{r, -}(l)$, all $L_{r, -}(\ell_{l-1}(w))$ such that $w\in W^r$ is of the first kind, and 
all $L_{r, -}(\ell_l(w))$ such that $w\in W^r$ is of the first kind. 

If $l\equiv 0 \pmod p$, then the simple composition factors of $H^0_{r, -}(l)$ are $L_{r, -}(l)$, $L_{r, -}(l-1)$, all $L_{r, -}(\ell_{l-1}(w))$ such that $w\in W^r$ is of the first kind, and 
all $L_{r, -}(\ell_l(w))$ such that $w\in W^r$ is of the second kind. 
\end{pr}
\begin{proof}
The proof is a simple modification of the combinatorial arguments in the proof of Proposition \ref{simplesosp}.
Observe that all simple composition factors of $H^0_{r, -}(l)$ have a multiplicity of at most one.

First assume that $p^r\leq  l < 2p^r$. We need to write the character $\chi(H^0_{r,-}(l))$ of $H^0_{r,-}(l)$ as a sum of characters $\chi(L_{r,-}(l_j))$ of simple $G_rT$-supermodules.

The character of $L_{r, -}(q)$ for $q\geq 0$ is obtained by truncating of the character of $L_-(q)$ in such a way that if $\chi(L_-(q))=\sum_{j\in J} x^{q-j}$, then 
$\chi(L_{r, -}(q))=\sum_{j\in J;0\leq j<2p^r} x^{q-j}$. Analogously, the character of $H^0_{r, -}(l)$ is obtained by truncating of the character of $H^0_-(l)$.
Because of the range for $k$ we easily observe that the weight spaces of $H^0_{r, -}(l)$ corresponding to weights $0, \ldots, l$ are one-dimensional.

Write $\chi(H^0_-(l))=\sum_{j=0}^{2l} x^{l-j}$ as a sum of $\chi(L_-(l_j))$ for appropriate $L_-(l)$, $L_-(l-1)$, $L_-(\ell_{l-1}(w))$ and $L_-(\ell_{l}(w))$ for $w\in W$ as in Proposition \ref{simplesosp}.
These simple supermodules $L_-(l_j)$ are completely characterized by the property that every weight space $0, \ldots, l$ belongs to exactly one $L_-(l_j)$.
Therefore each such corresponding $L_{r, -}(l_j)$ is a composition factor of $H^0_{r, -}(l)$. 
Since all $L_-(l_j)$ completely cover all weight spaces corresponding to weights $-l, \ldots, l$, every non-zero weight space of $H^0_{r, -}(l)$ belongs to the range $-l, \ldots, l$ and the characters of simples $L_{r, -}(l_j)$ and $H^0_{r, -}(l)$ are obtained by truncating of the characters of the corresponding $L_-(l_j)$ and $H^0_-(l)$, there cannot be any additional simple composition factors of $H^0_{r, -}(l)$ besides those already described.   
  
If $l$ is an arbitrary integer, we can find a multiple of $p^r$, say $tp^r$ such that $p^r<\tilde{l}=l-tp^r\leq 2p^r$. 
In this case $H^0_{r, -}(l)\cong H^0_{r, -}(\tilde{l})\otimes K_{tp^r}$, 
$L_{r, -}(l)\cong L_{r, -}(\tilde{l})\otimes K_{tp^r}$ and $L_{r, -}(l-1)\cong L_{r, -}(\tilde{l}-1)\otimes K_{tp^r}$.
Since $l$ and $\tilde{l}$ differ by a multiple of $p^r$, using Kummer's theorem, we infer that $\binom{l}{i}$ is a multiple of $p$ if and only if $\binom{\tilde{l}}{i}$ 
is a multiple of $p$ 
and $\binom{l-1}{i}$ is a multiple of $p$ if and only if $\binom{\tilde{l}-1}{i}$ is a multiple of $p$ for every $0\leq i<p^r$.
This implies that $L_{r, -}(\ell_l(w))\cong L_{r, -}(\ell_{\tilde l}(w))\otimes K_{tp^r}$. Using the above isomorphisms we extend the statement of the Proposition to arbitrary 
$l$.
\end{proof}

\begin{pr}\label{blocksospr}
Let $G=SpO(2|1)$. Then there are exactly $p$ blocks $\mathscr{B}_a$, for $0\leq a<p$ for every Frobenius thickening $G_rT$. 
The block $\mathscr{B}_a$ is described as $\mathscr{B}_a=\{l\in \mathbb{Z}| l \equiv a \pmod{2p}$ or $l\equiv 2p-1-a \pmod{2p}\}$. 
\end{pr}
\begin{proof}
First consider only pair of weights $l_1, l_2$ such that $p^r\leq l_1,l_2$.

Since the statements of Proposition \ref{simplesospr} differ from the statements of Proposition \ref{simplesosp} 
only in that $W^u$ is replaced by $W^r$,
the proof of our statement follows from the proof of Proposition \ref{blocksosp}, if we take into account that the linkage there is 
given by weights of the form $l-1$ or $\ell_{l-1}(w)$, where $w\in W_0\subset W^r$.

To complete the argument, it remains to apply the shift by a suitable multiple of $p^r$ so that both $\tilde{l}_1=l_1+tp^r, \tilde{l}_2=l_2+tp^r\geq p^r$.
\end{proof}

The simple composition factors of the kernel, image and cokernel of the morphism $\psi_{r, k}: H^0_{r, +}(-k)\to H^0_{r, -}(2p^r-k-1))$ from Lemma \ref{morphr} are described using the following Proposition.  

\begin{pr} The image of the morphism $\psi_{r,k}$ is isomorphic to $L_{r,-}(2p^r-k-1)$; the composition factors of the kernel (and cokernel) of the morphism
$\psi_{r,k}$ are all simple composition factors of $H^0_{r,+}(-k)$ (or $H^0_{r,-}(2p^r-k-1)$) except for $L_{r,-}(2p^r-k-1)$.
\end{pr}
\begin{proof}
Start by observing that $\chi(H^0_{r,+}(-k))=\sum_{j=-k}^{2p^r-k-1} x^j=\chi(H^0_{r,-}(2p^r-k-1))$, hence the simple composition factors of the domain and codomain of $\psi_{r,k}$ coincide and are described in Proposition \ref{simplesospr}.
Consequently, it is enough to describe simple composition factors of the image of $\psi_{r,k}$.

Assume $p^r\leq k <2p^r$. In this case $W^u=W^r$. Denote by $j_r$ that number $0\leq j_r <p^r$ that satisfies $k\equiv j_r \pmod{p^r}$. 
Because the morphism $\psi_{r,k}$ is uniquely determined by $x_{-1,-1}^{k-p^r}x_{-1,1}^{p^r-1}x_{-1,\overline{0}} \mapsto x_{1,1}^{2p^r-k-1}$ and 
the morphism $\psi_{-k,k-1-2j_r}: H^0_+(-k)\to H^0_-(k-1-2j_r)$ is uniquely determined by $x_{-1,-1}^{k-p^r}x_{-1,1}^{p^r-1}x_{-1,\overline{0}} \mapsto x_{1,1}^{k-1-2j_r}$, and 
$k-1-2j_r=2p^r-k-1=l$, the simple composition factors of the image of $\psi_{r,k}$ correspond to simple composition factors of the image of $\psi_{-k,k-1-2j_r}$ under a bijection preserving the highest weights.

If $j_r=0$, that is $k=p^r$, then by Proposition \ref{kercokk-1} the image of $\psi_{-k,k-1-2j_r}$ is isomorphic to the simple $L_{-}(k-1)$.
Therefore, the image of $\psi_{r,k}$ is the simple supermodule $L_{r,-}(2p^r-k-1)$. 

If $j_r>0$, then by Proposition \ref{kercokk-1-2j} the image of $\psi_{-k,k-1-2j_r}$ is isomorphic to a simple supermodule $L_-(k-1-2j_r)$ 
(because $\ell_k(\underbrace{\geq \ldots \geq}_r)=k-1-2j_r$). Therefore, the image of $\psi_{r,k}$ is the simple supermodule $L_{r,-}(2p^r-k-1)$. 

The description of the simple composition factors of the kernel and cokernel of $\psi_{r,k}$ now follows from Proposition \ref{simplesospr}.

If $k$ is arbitrary, then we find $t$ such that $p^r\leq \tilde{k}=k-tp^r<2p^r$ and use isomorphisms 
$H_{r, +}(-k)\cong H_{r, +}(-\tilde{k})\otimes K_{-tp^r}$ and $H^0_{r,-}(2p^r-k-1)\cong H^0_{r,-}(2p^r-\tilde{k}-1)\otimes K_{-tp^r}$ that relate the images of 
$\psi_{r,k}$ and $\psi_{r,\tilde{k}}$.
\end{proof}

\section{Strong linkage for $G_r T$}

\subsection{Preliminaries}

Following \cite{marzub}, for every $\lambda\in X(T)$ and $a=0, 1$ let 
$\mathcal{Z}'_{r, \mathcal{F}}(\lambda^a)$ and $\mathcal{Z}_{r, \mathcal{F}}(\lambda^a)$ denote
$ind^{G_r T}_{B^-(\mathcal{F})_r T} K^a_{\lambda}$ and $coind^{G_r T}_{B^+(\mathcal{F})_r T} K^a_{\lambda}$, respectively. 
If $\mathcal{F}=<1, \ldots, n, \overline{1}, \ldots, \overline{m}>$, then $\mathcal{Z}'_{r, \mathcal{F}}(\lambda^a)$ and $\mathcal{Z}_{r, \mathcal{F}}(\lambda^a)$ are denoted just by $\mathcal{Z}'_{r}(\lambda^a)$ and $\mathcal{Z}_{r}(\lambda^a)$.

Fix a maximal isotropic flag $\mathcal{F}$ and partially order all weights as $\mu\leq_{\mathcal{F}} \lambda$ if and only if $\lambda-\mu\in\sum_{\alpha\in\Phi^+(\mathcal{F})}\mathbb{N}\alpha$.
Let $\rho_0(\mathcal{F})$ and $\rho_1(\mathcal{F})$ denote $\frac{1}{2}(\sum_{\alpha\in\Phi^+(\mathcal{F})_0}\alpha)$ and $\frac{1}{2}(\sum_{\alpha\in\Phi^+(\mathcal{F})_1}\alpha)$, respectively, and $\rho(\mathcal{F})=\rho_0(\mathcal{F})-\rho_1(\mathcal{F})$. If $\mathcal{F}=<1, \ldots , n, \overline{1}, \ldots, \overline{m}>$, then 
$\rho_0(\mathcal{F}), \rho_1(\mathcal{F})$ and $\rho(\mathcal{F})$ are denoted just by $\rho_0, \rho_1$ and $\rho$, respectively. 
For a weight $\lambda\in X(T)$ let $\lambda<\mathcal{F}>$ denote the weight \[\lambda+(p^r -1)(\rho_0(\mathcal{F})-\rho_0)+(\rho_1(\mathcal{F})-\rho_1).\]
\begin{lm}\label{soclesandtops}
The supermodule $\mathcal{Z}'_{r, \mathcal{F}}(\lambda^a)$ has a simple socle and the supermodule $\mathcal{Z}_{r, \mathcal{F}}(\lambda^a)$ has a simple top. Moreover, these simple
supermodules are isomorphic to each other.
\end{lm}
\begin{proof}
Copy the proof of Lemmas 4.1 and 4.2 from \cite{marzub}.
\end{proof}

Denote by $\mathcal{L}_{r, \mathcal{F}}(\lambda^a)$ the simple socle of $\mathcal{Z}'_{r, \mathcal{F}}(\lambda^a)$ 
(which is isomorphic to the simple top of $\mathcal{Z}_{r, \mathcal{F}}(\lambda^a)$).
If $\mathcal{F}=<1, \ldots, n, \overline{1}, \ldots, \overline{m}>$, then $\mathcal{L}_{r, \mathcal{F}}(\lambda^a)$ is denoted just by $\mathcal{L}_{r}(\lambda^a)$.
As a byproduct of the proof of Lemma \ref{soclesandtops} we infer that either one of the conditions 
$\mathcal{Z}'_{r, \mathcal{F}}(\lambda^a)_{\mu}\neq 0, \mathcal{Z}_{r, \mathcal{F}}(\lambda^a)_{\mu}\neq 0$ or $\mathcal{L}_{r, \mathcal{F}}(\lambda^a)_{\mu}\neq 0$ implies 
$\mu\leq_{\mathcal{F}}\lambda$. 
Moreover, $\mathcal{Z}'_{r, \mathcal{F}}(\lambda^a)_{\lambda}=\mathcal{L}_{r, \mathcal{F}}(\lambda^a)_{\lambda}=\mathcal{Z}'_{r, \mathcal{F}}(\lambda^a)_{\mu}^{U^+(\mathcal{F})_r}$ is a one-dimensional supersubspace of parity $a$. 
Therefore every composition factor of $\mathcal{Z}'_{r, \mathcal{F}}(\lambda^a)$ (or $\mathcal{Z}_{r, \mathcal{F}}(\lambda^a)$, respectively) is isomorphic to 
$\mathcal{L}_{r, \mathcal{F}}(\mu^b)$, where $\mu\leq\lambda$ and $p(\lambda)+a\equiv p(\mu)+b \pmod 2$. 
Besides, $\mathcal{L}_{r, \mathcal{F}}(\lambda^a)$ has the multiplicity one in both $\mathcal{Z}'_{r, \mathcal{F}}(\lambda^a)$ and $\mathcal{Z}_{r, \mathcal{F}}(\lambda^a)$.

\begin{lm}\label{duality} For every $\lambda\in X(T)$ there are isomorphisms
\[\begin{aligned}&\mathcal{Z}_{r, \mathcal{F}}(\lambda)\simeq ind_{B^+(\mathcal{F})_r T}^{G_r T} K^{|(m+1)n|}_{\lambda -2((p^r -1)\rho_0(\mathcal{F}) +\rho_1(\mathcal{F}))}\\
&\simeq \Pi^{|(m+1)n|}(\mathcal{Z}_{r, \mathcal{F}} (2((p^r -1)\rho_0(\mathcal{F}) +\rho_1(\mathcal{F}))-\lambda)^*\end{aligned}\] 
and 
\[\begin{aligned}&\mathcal{Z}'_{r, \mathcal{F}}(\lambda)\simeq coind^{G_r T}_{B^-(\mathcal{F})_r T} K^{|(m+1)n|}_{\lambda -2((p^r -1)\rho_0(\mathcal{F}) +\rho_1(\mathcal{F}))} \\
&\simeq \Pi^{|(m+1)n|}(\mathcal{Z}'_{r, \mathcal{F}} (2((p^r -1)\rho_0(\mathcal{F}) +\rho_1(\mathcal{F}))-\lambda)^*.\end{aligned}\]
Their characters are given as  
\[\mathsf{ch}(\mathcal{Z}_{r, \mathcal{F}}(\lambda))=\mathsf{ch}(\mathcal{Z}'_{r, \mathcal{F}}(\lambda))=
e^{\lambda}\prod_{\alpha\in\Phi^+(\mathcal{F})_0}\frac{1-e^{-p^r\alpha}}{1-e^{-\alpha}}
\prod_{\alpha\in\Phi^+(\mathcal{F})_1}(1+e^{-\alpha}).\] 
Additionally, for every $\mu\in X(T)$ we have
\[\mathcal{Z}_{r, \mathcal{F}}(\lambda+p^r\mu)\simeq \mathcal{Z}_{r, \mathcal{F}}(\lambda)\otimes p^r\mu \text{ and } 
\mathcal{Z}'_{r, \mathcal{F}}(\lambda+p^r\mu)\simeq \mathcal{Z}'_{r, \mathcal{F}}(\lambda)\otimes p^r\mu.\]
\end{lm}
\begin{proof}
Copy the proof of Proposition 4.9 from \cite{marzub}.
\end{proof}
Let $\mathcal{F}=<1, \ldots, n, \overline{1}, \ldots, \overline{m}>$. In the same way as in \cite{marzub}, we obtain the following corollary.
\begin{cor}\label{atopof}
There is an isomorphism $\mathcal{Z}'_{r, -\mathcal{F}}(\lambda<-\mathcal{F}>)\simeq\mathcal{Z}_r(\lambda^{|(m+1)n|})$. Consequently, 
$\mathcal{L}_r(\lambda)$ is isomorphic (up to a parity shift) to the top of the supermodule $\mathcal{Z}'_{r, -\mathcal{F}}(\lambda<-\mathcal{F}>)$.
\end{cor}
\begin{lm}\label{characters}
For every maximal isotropic flag $\mathcal{F}$ the formal characters of $\mathcal{Z}'_{r, \mathcal{F}}(\lambda<\mathcal{F}>)$ and $\mathcal{Z}'_{r}(\lambda)$ (as well as the formal characters of $\mathcal{Z}_{r, \mathcal{F}}(\lambda<\mathcal{F}>)$ and $\mathcal{Z}_{r}(\lambda)$) coincide.
 
In particular, the simple supermodule $\mathcal{L}_r(\mu)$ is a composition factor of $\mathcal{Z}'_{r, \mathcal{F}}(\lambda<\mathcal{F}>)$ 
(or $\mathcal{Z}_{r, \mathcal{F}}(\lambda<\mathcal{F}>)$, respectively) if and only if it is a composition factor of $\mathcal{Z}'_{r}(\lambda)$ (or $\mathcal{Z}'_{r}(\lambda)$, respectively).
\end{lm}
\begin{proof}
For the first part, copy the proof of Lemma 4.10 from \cite{marzub}. Since the formal characters of simple supermodules $\mathcal{L}_r(\mu)$ are linearly independent, the second statement follows. 
\end{proof}
Lemma \ref{characters} has the following corollary.
\begin{cor}\label{morphismsbetweeninducedsupermodules}
For every $\lambda\in X(T)$ and every maximal isotropic flags $\mathcal{F}$ and $\mathcal{F}'$ there is an isomorphism
\[Hom_{G_r T}(\mathcal{Z}'_{r, \mathcal{F}}(\lambda<\mathcal{F}>), \mathcal{Z}'_{r, \mathcal{F}'}(\lambda<\mathcal{F}'>))\simeq K.\]
\end{cor}
\begin{proof}
Use the Frobenius reciprocity law and the observation that $\lambda<\mathcal{F}'>$ is the largest weight of
$\mathcal{Z}'_{r, \mathcal{F}}(\lambda<\mathcal{F}>)|_{B^-(\mathcal{F}')_r T}$ with respect to the partial order $\leq_{\mathcal{F}'}$.
\end{proof}
The proof of the following Proposition follows from Lemmas \ref{BorelsandParabolics1},  \ref{BorelsandParabolics2} and \ref{BorelsandParabolics3}, and from their analogs for $SpO(2n|2m)$ as well.
\begin{pr}\label{chainofBorels}
There is a chain of Borel supersubgroups $B_0, B_1, \ldots , B_l$ such that $B_0=B^-, B_l=B^+$, and $B_i, B_{i+1}$ are even or odd adjacent for every $0\leq i\leq l-1$.
\end{pr}
\begin{proof} We describe explicitly the flags corresponding to the Borel supersubgroups in the Proposition. 

We start with the flag $\mathcal{F}_0=(1, \ldots, n, \overline{1}, \ldots , \overline{m})$. First, we use a sequence of $n$ transpositions of neighboring terms to move $\overline{1}$ through all elements from the list $\{1, \ldots, n\}$ as in 
Lemma \ref{BorelsandParabolics1}. 
Next, we use $n$ transpositions of neighboring terms to move $\overline{2}$ through all elements from the list $\{1, \ldots, n\}$ and repeat this procedure 
for $\overline{3}, \ldots, \overline{m}$. After $mn$ steps we obtain the flag $\mathcal{F}_{mn}=(\overline{1}, \ldots , \overline{m}, 1, \ldots, n)$. 

Next, we replace the rightmost element $n$ by its negative $-n$ as in Lemma \ref{BorelsandParabolics2} and move it using $m+n-1$ transpositions of neighboring terms to the first position. Repeating the same procedure with $n-1$ through $1$, after $n$ replacements and $n(m+n-1)$ transpositions we obtain the flag $\mathcal{F}_{n(2m+n)}=(-1, \ldots, -n, \overline{1}, \ldots, \overline{m})$. Finally, we replace the rightmost element $\overline{m}$ by its negative $-\overline{m}$ as in Lemma \ref{BorelsandParabolics3} and move it 
using $k-1$ transpositions to the position immediately right of $-n$. Then we repeat this for $\overline{k-1}$ through $\overline{1}$ and after $m$ replacements and $m^2-m$ transpositions we obtain $\mathcal{F}_{(m+n)^2}=-\mathcal{F}_0$.
\end{proof}

\subsection{Doty's approach}

Now one can simulate the Doty's idea from \cite{sdoty} as follows. Denote by $B_i=B^-(\mathcal{F}_i)$ the Borel supersubgroups from Proposition \ref{chainofBorels} so that 
$\mathcal{F}_0=< 1, \ldots , n, \overline{1}, \ldots , \overline{m}>$ and $\mathcal{F}_l=-\mathcal{F}_0$. For every $1\leq i\leq l$ denote by
\[f_i : \mathcal{Z}'_{r, \mathcal{F}_i}(\lambda<\mathcal{F}_i>)\to \mathcal{Z}'_{r, \mathcal{F}_{i-1}}(\lambda<\mathcal{F}_{i-1}>)\]
a non-zero homomorphism from Corollary \ref{morphismsbetweeninducedsupermodules}. Set $f=f_1\cdot\ldots\cdot f_{l-1} f_l$. If $f\neq 0$, then the image of $f$ coincides with the socle $\mathcal{L}_r(\lambda)$ of $\mathcal{Z}'_{r, \mathcal{F}_0}(\lambda<\mathcal{F}_0>)=\mathcal{Z}'_{r}(\lambda)$ by Corollary \ref{atopof}.

Assume that $\mathcal{L}_r(\mu)$ (or its parity shift) is a composition factor of $\mathcal{Z}'_{r}(\lambda)$ and $\mu\neq\lambda$. Then $\mathcal{L}_r(\mu)$ is a composition factor of some $\ker f_i$.
Let $P$ be a minimal parabolic supersubgroup that contains both $B_i=B^+(\mathcal{F}_i)$ and 
$B_{i-1}=B^+(\mathcal{F}_{i-1})$.  Since the quotient $G_r T/P_r T\simeq G_r/P_r$ is affine, the functor
$ind ^{G_r T}_{P_r T}$ is (faithfully) exact. Thus $\ker f_i=ind ^{G_r T}_{P_r T} (\ker \tilde{f}_i)$, where $\tilde{f}_i : ind ^{P_r T}_{(B_i)_r T} K_{\lambda<\mathcal{F}_i>} \to ind ^{P_r T}_{(B_{i-1})_r T} K_{\lambda <\mathcal{F}_{i-1}>}$.

If $B_i$ and $B_{i-1}$ are even adjacent or odd adjacent via an isotropic root, then the switch from $B_{i-1}$ to $B_{i}$ can be treated as in Proposition 6.2 from \cite{marzub} 
(see also Lemma 6.1 therein). We leave this part as an exercise for the reader.

Consider the case when $B_i$ and $B_{i-1}$ are odd adjacent via a non-isotropic simple root $\alpha=\delta_j\in\Phi^+ (\mathcal{F}_{i-1})$. In this case $G$ should be $SpO(2n|2m+1)$. 
Then $\lambda<\mathcal{F}_i>=\lambda<\mathcal{F}_{i-1}>-(2p^r -1)\alpha$. Combining Lemma \ref{BorelsandParabolics2} with Lemma 10.4 and Corollary 10.2 from \cite{zub1} 
(see also Proposition 11.5(1) therein) one concludes that the morphism $\tilde{f}_i$ can be identified with
\[\psi_{r, k} : H_{r, +}(-k)\to H_{r, -}(2p^r -k-1),\]
where $H_{r, +}(-k)\simeq ind^L_{\mathcal{B}^+} K_{-k}$ and $H_{r, -}(2p^r -k-1)\simeq ind^L_{\mathcal{B}^-} K_{2p^r -k-1}$. 
Here $L\simeq SpO(2|1)$, $\mathcal{B}^+=B_i\cap L$ is an "upper triangular" Borel supersubgroup of $L$, 
$\mathcal{B}^-=B_{i-1}\cap L$ is a "lower triangular" Borel supersubgroup of $L$ and $-k=(\lambda<\mathcal{F}_i>, \alpha)$.

Then the supermodule $\mathcal{L}_r(\mu)$ (or its parity shift) is a composition factor of some 
$ind^{G_r T}_{P_r T} L_{r, -}(l)$, where $L_{r, -}(l)$ is a composition factor of both kernel and and cokernel of $\psi_{r, k}$.
We have
\[ind^{G_r T}_{P_r T} L_{r, -}(l)\subseteq ind^{G_r T}_{P_r T} H^0_{r, -}(l)\simeq ind^{G_r T}_{P_r T} ind ^{P_r T}_{(B_{i-1})_r T} K_{\lambda' <\mathcal{F}_{i-1}>}=
\mathcal{Z}'_{r, \mathcal{F}_{i-1}}(\lambda'<\mathcal{F}_{i-1}>),\]
where $\lambda'<\mathcal{F}_{i-1}>$ is obtained from $\lambda<\mathcal{F}_{i-1}>$ by replacing
its $j$-th coordinate by $l$. 
Using Lemma \ref{characters} we derive that $\mathcal{L}_r(\mu)$ is a composition factor of $\mathcal{Z}'_r(\lambda')$.

\begin{lm}\label{valuesoftheform}
The following are values of the bilinear form involving the root $\alpha$.
$(\rho_0(\mathcal{F}_{i-1}), \alpha)=1, (\rho_1(\mathcal{F}_{i-1}), \alpha)=\frac{1}{2}$ and $(\rho_0, \alpha)=n-j+1, (\rho_1, \alpha)=m+\frac{1}{2}$.
\end{lm}
\begin{proof} According to Lemma \ref{chainofBorels}, $\mathcal{F}_{i-1}=(a_1, \ldots, a_{m+n})$, where $a_{m+n}=\delta_j$. Using Lemma \ref{BorelsandParabolics1} one 
obtains the first two equalities. The remaining equalities can be calculated directly. 
\end{proof}
Lemma \ref{valuesoftheform} implies
\[(\lambda<\mathcal{F}_{i-1}>, \alpha)=(\lambda +p^r(\rho_0(\mathcal{F}_{i-1})-\rho_0)-\rho(\mathcal{F}_{i-1})+\rho, \alpha)=(\lambda+\rho, \alpha)-p^r(n-j)-\frac{1}{2}.
\]

Now we can formulate the strong linkage for $G_r T$ as follows.
\begin{tr}\label{stronglinkageforG_r T} 
If $\mathcal{L}_r(\mu)$ (or $\Pi\mathcal{L}_r(\mu)$) is a composition factor of $\mathcal{Z}'_r(\lambda)$ and $\mu\neq\lambda$, then there is a sequence $\lambda=\lambda_0, \lambda_1, \ldots, \lambda_t=\mu$ such that for each $1\leq i\leq t$ one of the following alternatives holds.
\begin{enumerate}
\item $\lambda_{i}=\lambda_{i-1}-\alpha_i$, where $\alpha_i$ is a positive odd isotropic root such that $p|(\lambda_{i-1}+\rho, \alpha_i)$; 
\item $\lambda_i=\lambda_{i-1}-(l-l')\alpha_i$, where $\alpha_i$ is a positive odd non-isotropic root , $l$ is the smallest non-negative residue of $(\lambda+\rho, \alpha_i)-\frac{1}{2}$ modulo $p^r$ and
$l'$ is one of the weights $l-1$, $\ell_l(w)$ or $\ell_{l-1}(w)$ listed in Proposition \ref{simplesospr};
\item $\lambda_i <\lambda_{i-1}$ and $\lambda_i$ is $(\alpha_i, r)$-linked to $\lambda_{i-1}$, where $\alpha_i$ is a positive even root. 
\end{enumerate}
Moreover, the case (2) only appears for $G=SpO(2n|2m+1)$.
\end{tr}
\begin{proof}
The proof extends the ideas presented in the proof of Proposition 6.2 from \cite{marzub}. The cases $(1)$ and $(3)$, respectively, correspond to Borel supersubgroups $B_{i-1}$ and $B_i$ that are adjacent via a positive isotropic odd root and a positive even root, respectively; in both cases we use the same arguments as in \cite{marzub}. In the case
$(2)$ we observe that 
\[(\lambda<\mathcal{F}_{i-1}>, \alpha)=(\lambda +p^r(\rho_0(\mathcal{F}_{i-1})-\rho_0)-\rho(\mathcal{F}_{i-1})+\rho, \alpha)=(\lambda+\rho, \alpha)-p^r(n-j)-\frac{1}{2}\]
and use Proposition \ref{simplesospr}.
\end{proof}
\begin{rem}\label{replacing}
Using Lemma \ref{duality} we can show that the statement of Theorem \ref{stronglinkageforG_r T} remains true if $\mathcal{Z}'_r(\lambda)$ is replaces by $\mathcal{Z}_r(\lambda)$.
\end{rem}

\section{Linkage for $G$}

If $\mathcal{F}=< 1, \ldots , n, \overline{1}, \ldots, \overline{m} >$ is the standard flag, then we write $L(\lambda)$ for $L_{\mathcal{F}}(\lambda)$, 
$V(\lambda)$ for $V_{\mathcal{F}}(\lambda)$ and so on.

Using Theorem \ref{stronglinkageforG_r T} and Proposition 4.4 from \cite{shuw}, we extend the proof of Proposition 7.2 from \cite{marzub} to ortho-symplectic supergroups as follows.
\begin{pr}\label{simplelinkagefor G}
Assume that $L(\mu)$ (or its parity shift) is a composition factor of $V(\lambda)$. Then 
there is a positive integer $r$ and a sequence $\lambda=\lambda_0, \lambda_1, \ldots, \lambda_s=\mu$ such that for each $1\leq i\leq s$ one of the following alternatives holds.
\begin{enumerate}
\item $\lambda_{i}=\lambda_{i-1}-\alpha_i$, where $\alpha_i$ is a positive odd isotropic root such that $p|(\lambda_{i-1}+\rho, \alpha_i)$;
\item $\lambda_i=\lambda_{i-1}-(l-l')\alpha_i$, where $\alpha_i$ is a positive odd non-isotropic root, $l$ is the smallest non-negative residue of $(\lambda_i+\rho, \alpha_i)-\frac{1}{2}$ modulo $p^r$ and
$l'$ is one of the weights $l-1$, $\ell_l(w)$ or $\ell_{l-1}(w)$ listed in Proposition \ref{simplesospr};
\item $\lambda_i <\lambda_{i-1}$ and $\lambda_i$ is $(\alpha_i, r)$-linked to $\lambda_{i-1}$, where $\alpha_i$ is a positive even root. 
\end{enumerate}
As in Theorem \ref{stronglinkageforG_r T}, the case (2) only appears for $G=SpO(2n|2m+1)$.
\end{pr}
We say that the weights $\lambda$ and $\mu$ are $r$-simply-isotropically-odd-linked if there is an odd isotropic root $\alpha$ such that $p|(\lambda+\rho, \alpha)$ and 
$\mu=\lambda-\alpha$. 
Similarly, we say that the weights $\lambda$ and $\mu$ are $r$-simply-non-isotropically-odd-linked if there is a non-isotropical odd root $\alpha$ such that either 
$\mu=\lambda-(l-l')\alpha$, where $l$ is the smallest non-negative residue of $(\lambda+\rho, \alpha)-\frac{1}{2}$ modulo $p^r$ and
$l'$ is one of the weights $l-1$, $\ell_l(w)$ or $\ell_{l-1}(w)$ listed in Proposition \ref{simplesospr}, or $\mu=\lambda-(l-l')\alpha$, where $l$ is the smallest non-negative residue of $(\mu+\rho, \alpha)-\frac{1}{2}$ modulo $p^r$ and
$l'$ is one of the weights $l-1$, $\ell_l(w)$ or $\ell_{l-1}(w)$ listed in Proposition \ref{simplesospr}.

The following theorem is a consequence of Proposition \ref{simplelinkage} and Proposition \ref{simplelinkagefor G}.
\begin{tr}\label{final}
If weights $\lambda$ and $\mu$ belong to the same block, then there is a sequence $\lambda=\lambda_0, \lambda_1, \ldots, \lambda_s=\mu$ such that for each $1\leq i\leq s$ 
the neighbors $\lambda_i$ and $\lambda_{i-1}$ are either $r$-simply-isotropically/non-isotropically-odd-linked or $(\alpha, r)$-linked for an even root $\alpha$ and a 
positive number $r$. 
\end{tr}


\begin{thebibliography}{99}
\bibitem{chengwang} S.~J.~Cheng and W.~Wang, {\em Dualities and Representations of Lie Superalgebras}, Graduate Studies in Mathematics 144, American Math. Soc., Providence,
2012.

\bibitem{sdoty} S.~Doty, {\em The strong linkage principle}. Amer. J. Math. 111 (1989), no. 1, 135--141.

\bibitem{jan} J.~C.~Jantzen,  {\em Representations of algebraic groups}. Second edition. Mathematical Surveys and Monographs, 107. American Mathematical Society, Providence, RI, 2003.

\bibitem{marzub} F.~Marko and A.N.~Zubkov, {\em Blocks for general linear supergroup $GL(m|n)$},
submitted to Transformation Groups (see also arXiv: 1507.05027).

\bibitem{amas1} A.~Masuoka, {\em Harish-Chandra pairs for algebraic affine supergroup schemes over an arbitrary field}, Transform. Groups 17(2012), no. 4, 1085--1121 (see also arXiv: 1111.2387).

\bibitem{shuw} B.~Shu and W.~Wang, {\em Modular representations of ortho-symplectic supergroups}, Proc.London\ Math.Soc., (3) 96(2008),
251-271.

\bibitem{zubul} P.~Ulayshev and A.~N~.Zubkov, {\em Solvable and unipotent supergroups}, 
Algebra and Logic (Russian), 53, N 3 (2014), 323--339.

\bibitem{zub1} A.~N.~Zubkov, {\em Some homological properties of $GL(m|n)$ in arbitrary characteristic}, to appear in J. Algebra Appl. doi:10.1142/S021949881650119X (see also arXiv:1405.3890)

\bibitem{zub2} A.~N.~Zubkov, {\em Affine quotients of supergroups}. Transform.\ Groups, 14 (2009), no. 3, 713--745.

\bibitem{zub3} A.~N.~Zubkov, 
{\em Some properties of general linear supergroups and of Schur superalgebras}. Algebra\ Logic, 45 (2006), no. 3, 147--171. 

\end{thebibliography}
\end{document}